\documentclass[leqno,a4paper,12pt]{article}
\usepackage{amsmath,amsthm,amssymb}
\usepackage[T1]{fontenc}
\usepackage{hyperref}
\usepackage{enumerate}
\usepackage{bbm}
\usepackage{bm}
\usepackage{mathrsfs}
\usepackage{nicefrac}
\usepackage[a4paper]{geometry}
\usepackage{xcolor}
\usepackage{tikz}
\tikzset{every picture/.style={line width=0.75pt}} 

\DeclareMathOperator{\lcm}{lcm}
\usepackage[sort&compress,numbers]{natbib}
\newtheorem{Th}{Theorem}[section]
\newtheorem{Prop}[Th]{Proposition}
\newtheorem{Lemma}[Th]{Lemma}
\newtheorem{Cor}[Th]{Corollary}
\newtheorem{Thx}{Theorem}

\newtheorem{Propx}[Thx]{Proposition}
\theoremstyle{definition}
\newtheorem{Remark}[Th]{Remark}
\newtheorem{Def}[Th]{Definition}
\newtheorem{Example}[Th]{Example}

\newcommand{\divides}{\mid}
\newcommand{\ndivides}{\nmid}
\newcommand{\Per}{\operatorname{Per}\nolimits}

\newcommand{\mob}{\boldsymbol{\mu}}
\newcommand{\cf}{{\cal F}}
\newcommand{\cm}{{\cal M}}
\newcommand{\cM}{{\cal M}}
\newcommand{\ov}{\overline}
\newcommand{\Q}{\mathbb{Q}}
\newcommand{\R}{{\mathbb{R}}}

\newcommand{\Z}{{\mathbb{Z}}}
\newcommand{\N}{{\mathbb{N}}}
\newcommand{\PP}{{\mathbb{P}}}
\newcommand{\OK}{\mathcal{O}_K}
\newcommand{\G}{\mathbb{G}}
\newcommand{\cA}{\mathcal{A}}
\newcommand{\cC}{\mathcal{C}}
\newcommand{\HH}{{\mathbb{H}}}
\newcommand{\vep}{\varepsilon}

\newcommand{\sB}{\mathscr{B}}

\newcommand{\raz}{\mathbbm{1}}

\makeatletter
\def\moverlay{\mathpalette\mov@rlay}
\def\mov@rlay#1#2{\leavevmode\vtop{%
   \baselineskip\z@skip \lineskiplimit-\maxdimen
   \ialign{\hfil$\m@th#1##$\hfil\cr#2\crcr}}}
\newcommand{\charfusion}[3][\mathord]{
    #1{\ifx#1\mathop\vphantom{#2}\fi
        \mathpalette\mov@rlay{#2\cr#3}
      }
    \ifx#1\mathop\expandafter\displaylimits\fi}
\makeatother

\newcommand{\bigcupdot}{\charfusion[\mathop]{\bigcup}{\cdot}}
\def\supp{{\rm supp }}
\author{Aurelia Dymek \and Stanis\l{}aw Kasjan \and Joanna Ku\l{}aga-Przymus}
\title{Minimality of $\mathfrak{B}$-free systems in number fields}
\begin{document}
\maketitle
\abstract{Let $K$ be a finite extension of $\Q$ and $\OK$ be its ring of integers. Let $\mathfrak{B}$ be a primitive collection of ideals in $\OK$. We show that any $\mathfrak{B}$-free system is essentially minimal. Moreoever, the $\mathfrak{B}$-free system is minimal if and only if the characteristic function of $\mathfrak{B}$-free numbers is a Toeplitz sequence. Equivalently, there are no ideal $\mathfrak{d}$ and no infinite pairwise coprime collection of ideals $\cC$ such that $\mathfrak{d}\cC\subseteq\mathfrak{B}$. Moreover, we find a periodic structure in the Toeplitz case. Last but not least, we describe the restrictions on the cosets of ideals contained in unions of ideals.}
\tableofcontents
\section{Background and main results}\label{se1}

Given a set $\sB$ of natural numbers, we say that an~integer $n$ is~$\sB$-free, if no number in~$\sB$ divides $n$. The set of $\sB$-free integers is~denoted by $\cf_\sB$ and its complement, which is the corresponding set of multiplies, by $\cM_{\sB}$. These sets were studied intensively from the number theoretic viewpoint (see \cite{MR1414678}, and \cite{MR3803141} for more references). Our approach is dynamical and it is the continuation of the line of research initiated by Sarnak. Let $\eta\in \{0,1\}^\Z$ stand for the~characteristic function of~$\cf_\sB$ and let $X_\eta$ be the orbit closure of $\eta$ with respect to~the~left shift $\sigma$~on $\{0,1\}^\Z$. The resulting subshift $(X_\eta,\sigma)$ is~called a $\sB$-free system. In 2010 Sarnak suggested to study the so-called square-free system, i.e.\ the $\sB$-free system for $\sB$ being the set of squares of all primes. The underlying motivation was to understand the random nature of the arithmetic M\"obius function $\mob$ -- note that $\mob^2$ is the characteristic function of the set of square-free integers.

As each set $b\Z$ is~a subgroup of~the additive group $\Z$ and also an~ideal of~the ring $\Z$, there are two natural ways of generalizing the notion of a $\sB$-free system to higher dimensions. Namely, as the counterpart of the set of multiples $\cM_{\sB}$ we choose
\begin{itemize}
\item countable unions of~sublattices of~$\Z^m$,
\item countable unions of~ideals in the~rings of~integers $\OK$ of~algebraic number fields $K$.
\end{itemize}
Similarly to~\cite{MR4251829}, we consider a~collection $\mathfrak{B}$ of non-zero ideals in~$\OK$. Put $\mathcal{F}_\mathfrak{B}:=\OK\setminus\bigcup_{\mathfrak{b}\in\mathfrak{B}}\mathfrak{b}$ and let $X_{\eta}$ be the orbit closure of $\eta=\raz_{\mathcal{F}_\mathfrak{B}}\in\{0,1\}^{\OK}$ under the multidimensional shift $(S_g)_{g\in\OK}$. The pair $(X_\eta,(S_g)_{g\in\OK})$ is called a~\emph{$\mathfrak{B}$-free system}.

Recall that:
\begin{enumerate}
\item[(A)]
Each $\sB$-free subshift is essentially minimal, i.e.\ $(X_\eta,\sigma)$ has a unique minimal subset (which is the orbit closure of a Toeplitz sequence ({\cite[Theorem A]{MR3803141}})).\label{ama}
\item[(B)]
The $\sB$-free subshift $(X_\eta,\sigma)$ is minimal if and only if it is a Toeplitz system ({\cite[Corollary 1.4]{MR3803141}}).
\end{enumerate}
Then, in~\cite{MR3947636} the following result related to (B) was shown:
\begin{enumerate}
\item[(B')]
The minimality of $(X_\eta,\sigma)$ implies that $\eta$ itself is a Toeplitz sequence, under the extra assumption that $\mathscr{B}$ is taut.
\end{enumerate}
Finally, Keller~\cite{MR4280951} showed the following:
\begin{enumerate}
\item[(A')]
The unique minimal subset of $(X_\eta,\sigma)$ is also a $\mathscr{B}$-free system corresponding to a set $\mathscr{B}^*$, which is a certain modification of $\mathscr{B}$.
\end{enumerate}
For more details on (A),(A'),(B),(B'), see Section~\ref{se3}.

The main goal of this paper is to prove analogous results to the above in the multidimensional case described above. Before we  formulate the results, let us introduce some notation. Given a collection  $\mathfrak{B}$ of ideals in $\OK$, let
\begin{align}
\begin{split}\label{A}
\mathfrak{D}=&\left\{\mathfrak{d}\subseteq \OK : \mathfrak{d} \text{ is a non-zero ideal\footnotemark{} such that }\mathfrak{d}\mathcal{C}\subseteq\mathfrak{B}\right.\\
&\left.\text{ for some infinite pairwise coprime collection $\mathcal{C}$ of ideals in }\OK\right\},
\end{split}
\end{align}
\footnotetext{The ideal $\mathfrak{d}$ can be equal to $\OK$.}
$\mathfrak{B}^*=(\mathfrak{B}\cup \mathfrak{D})^{prim}$ and $\eta^*=\raz_{\mathcal{F}_{\mathfrak{B}^*}}$.
\begin{Thx}[{see~\cite[Corollary 5]{MR4280951} for $\OK=\Z$}]\label{minimal2}
Let $\mathfrak{B}$ be a primitive collection of ideals in $\OK$. Then $\eta^*$ is an $\OK$-Toeplitz array and $X_{\eta^*}$ is a unique minimal subset of~$X_\eta$. Moreover, $\eta^*\leq\eta$.
\end{Thx}
\begin{Thx}\label{equivalent}
Let $\mathfrak{B}$ be a primitive collection of ideals in~$\OK$. The following are equivalent:
\begin{enumerate}[(i)]
\item $(X_\eta,(S_g)_{g\in\OK})$ is minimal,\label{min}
\item $\eta$ is an $\OK$-Toeplitz array different from $\mathbf{0}$, where $\mathbf{0}_g=0$ for any $g\in\mathcal{O}_K$,\label{top}
\item $\mathfrak{D}=\emptyset$,\label{arith}
\item $X_\eta\subseteq Y$, where $Y=\{x\in\{0,1\}^{\OK};\ |\operatorname{supp }x \bmod \mathfrak{b}|=N(\mathfrak{b})-1\text{ for any }\mathfrak{b}\in\mathfrak{B}\}$ ($N(\mathfrak{b})$ stands for the cardinality of the quotient $\OK/\mathfrak{b}$).\label{Y}
\end{enumerate}
\end{Thx}

Notice that even in the one-dimensional case this is a strenghtening of the earlier results. Namely, our methods allow us to skip the technical (as it turns out) assumption that $\mathscr{B}$ is taut in order to deduce that $\eta$ is a Toeplitz sequence whenever $(X_\eta,\sigma)$ is minimal. The proofs substantially differ from the earlier ones, even in the one-dimensional case.

The proof of Theorem~\ref{minimal2} relies on the following result, which, in our opinion, catches the essence of what is happening in the Toeplitz case.
\begin{Thx}\label{dwyz}
Let $\mathfrak{B}$ be primitive. Then the following are equivalent:
\begin{enumerate}[(i)]
\item $\mathfrak{B}=\mathfrak{B}^*$,
\item $\mathfrak{D}=\emptyset$,
\item $\eta$ is a Toeplitz sequence different from $\bf{0}$.
\end{enumerate}
Moreover, we have
\[
\mathfrak{D}^*:=\{\mathfrak{d} :  \mathfrak{d}\cC\subset\mathfrak{B}^* \text{ for some infinite pairwise coprime set }\cC\}=\emptyset,
\]
i.e.\ $\mathfrak{B}^*=(\mathfrak{B}^*)^*$ and $\eta^*$ is Toeplitz.
\end{Thx}

Let us now discuss the remaining important ingredients in more details.
Clearly, what prevents $\eta$ from being Toeplitz, is the presence of non-periodic positions. We provide a description of the set of such positions.
\begin{Propx}\label{propA}
We have
\begin{equation}\label{teza}
\OK\setminus \bigcup_{\mathfrak{s}\subset \OK}\Per(\eta,\mathfrak{s})=\mathcal{M}_{\mathfrak{D}}\cap \mathcal{F}_{\mathfrak{B}}=\OK\setminus\bigcup_{i\geq1}\Per(\eta,\mathfrak{s}_i),
\end{equation}
where $\mathfrak{D}$ is given by~\eqref{A}, $(\mathfrak{s}_i)_{i\ge 1}$ is an ideal  filtration of $\OK$ and summation is over all nonzero ideals $\mathfrak{s}\subseteq \OK$.
\end{Propx}
``As a bonus'', Proposition~\ref{propA} allows us to find a periodic structure of $\eta$ in the Toeplitz case.
\begin{Thx}\label{thxC}
Let $\mathfrak{B}$ be a primitive collection of ideals in~$\OK$. Suppose $\mathfrak{S}_1\subset \mathfrak{S}_2\subset\dots\nearrow\mathfrak{B}$ is a saturated filtration of $\mathfrak{B}$ by finite collections of ideals in $\OK$. Assume that $\eta$ is an $\OK$-Toeplitz array. Then $(\bigcap_{\mathfrak{b}\in\mathfrak{S}_i}\mathfrak{b})_{i\geq1}$ is a periodic structure of $\eta$.
\end{Thx}

The last key ingredient of this paper is a result of arithmetic nature, on cosets of ideals (``arithmetic progressions'') contained in the set of multiples  in case when $\eta$ is Toeplitz.
\begin{Thx}[see Proposition \ref{druggi2Toeplitz} and the preceeding comments]\label{Thm_SK}
Let $\mathfrak{B}$ be a primitive collection of ideals in $\OK$ such that $\mathfrak{D}=\emptyset$. Let $r\in\OK$ and  let $\mathfrak{a}\subset\OK$ be an ideal. If
\[
r+\mathfrak{a}\subseteq\cm_\mathfrak{B},
\]
then for some $\mathfrak{b}\in\mathfrak{B}$
\[
\gcd((r),\mathfrak{a})=(r)+\mathfrak{a}\subseteq\mathfrak{b}.
\]
\end{Thx}

\paragraph{Structure of the paper}
The relation between the results stated above and the location of their proofs in the paper is as follows (the arrow from ``result 1'' to ``result 2'' means that ``result 1'' is used in the proof of ``result 2''):
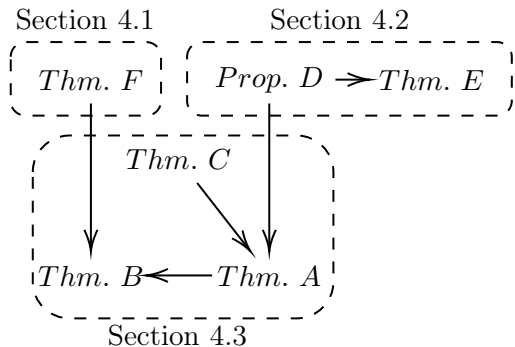
\begin{center}
\begin{tikzpicture}[x=0.75pt,y=0.75pt,yscale=-1,xscale=1]

\draw  [dash pattern={on 4.5pt off 4.5pt}] (58,246.5) .. controls (58,236.34) and (66.24,228.1) .. (76.4,228.1) -- (189.6,228.1) .. controls (199.76,228.1) and (208,236.34) .. (208,246.5) -- (208,301.7) .. controls (208,311.86) and (199.76,320.1) .. (189.6,320.1) -- (76.4,320.1) .. controls (66.24,320.1) and (58,311.86) .. (58,301.7) -- cycle ;
\draw  [dash pattern={on 4.5pt off 4.5pt}] (135,188.5) .. controls (135,184.41) and (138.31,181.1) .. (142.4,181.1) -- (289.6,181.1) .. controls (293.69,181.1) and (297,184.41) .. (297,188.5) -- (297,210.7) .. controls (297,214.79) and (293.69,218.1) .. (289.6,218.1) -- (142.4,218.1) .. controls (138.31,218.1) and (135,214.79) .. (135,210.7) -- cycle ;
\draw  [dash pattern={on 4.5pt off 4.5pt}] (47,188.5) .. controls (47,184.41) and (50.31,181.1) .. (54.4,181.1) -- (114.6,181.1) .. controls (118.69,181.1) and (122,184.41) .. (122,188.5) -- (122,210.7) .. controls (122,214.79) and (118.69,218.1) .. (114.6,218.1) -- (54.4,218.1) .. controls (50.31,218.1) and (47,214.79) .. (47,210.7) -- cycle ;

\draw (87,200) node    {$Thm.\ \ref{Thm_SK}$};
\draw (87,298.5) node    {$Thm.\ \ref{equivalent}$};
\draw (176,200) node    {$Prop.\ \ref{propA}$};
\draw (257,200) node    {$Thm.\ \ref{thxC}$};
\draw (130.75,240) node    {$Thm.\ \ref{dwyz}$};
\draw (176,298.5) node    {$Thm.\ \ref{minimal2}$};
\draw (175,163.1) node [anchor=north west][inner sep=0.75pt]   [align=left] {Section 4.2};
\draw (94,322.1) node [anchor=north west][inner sep=0.75pt]   [align=left] {Section 4.3};
\draw (48,163.1) node [anchor=north west][inner sep=0.75pt]   [align=left] {Section 4.1};
\draw    (209,200) -- (224,200) ;
\draw [shift={(226,200)}, rotate = 180] [color={rgb, 255:red, 0; green, 0; blue, 0 }  ][line width=0.75]    (10.93,-3.29) .. controls (6.95,-1.4) and (3.31,-0.3) .. (0,0) .. controls (3.31,0.3) and (6.95,1.4) .. (10.93,3.29)   ;
\draw    (87,212) -- (87,284.5) ;
\draw [shift={(87,286.5)}, rotate = 270] [color={rgb, 255:red, 0; green, 0; blue, 0 }  ][line width=0.75]    (10.93,-3.29) .. controls (6.95,-1.4) and (3.31,-0.3) .. (0,0) .. controls (3.31,0.3) and (6.95,1.4) .. (10.93,3.29)   ;
\draw    (116.5,298.5) -- (147.5,298.5) ;
\draw [shift={(114.5,298.5)}, rotate = 0] [color={rgb, 255:red, 0; green, 0; blue, 0 }  ][line width=0.75]    (10.93,-3.29) .. controls (6.95,-1.4) and (3.31,-0.3) .. (0,0) .. controls (3.31,0.3) and (6.95,1.4) .. (10.93,3.29)   ;
\draw    (176,212) -- (176,284.5) ;
\draw [shift={(176,286.5)}, rotate = 270] [color={rgb, 255:red, 0; green, 0; blue, 0 }  ][line width=0.75]    (10.93,-3.29) .. controls (6.95,-1.4) and (3.31,-0.3) .. (0,0) .. controls (3.31,0.3) and (6.95,1.4) .. (10.93,3.29)   ;
\draw    (140.03,252) -- (165.49,284.92) ;
\draw [shift={(166.72,286.5)}, rotate = 232.28] [color={rgb, 255:red, 0; green, 0; blue, 0 }  ][line width=0.75]    (10.93,-3.29) .. controls (6.95,-1.4) and (3.31,-0.3) .. (0,0) .. controls (3.31,0.3) and (6.95,1.4) .. (10.93,3.29)   ;

\end{tikzpicture}

\end{center}

In Section~\ref{se2} we introduce and recall the basic objects from the following areas: number fields and ideals, $\mathfrak{B}$-free integers on number fields, Toeplitz sequences. A part of the subsection on Toeplitz sequences related to the notion of essential periods is new and may be of an independent interest. Then, in Section~\ref{se3}, we discuss one dimensional $\sB$-free systems. Even though the results included there are a special case of their multidimensional counterparts covered in Section~\ref{se4}, we decided to keep it this way, as some readers may be interested in one-dimensional case only. Moreover, the multidimensional objects are much more technical and Section~\ref{se4} may be quite hard to digest otherwise. The structure of Section~\ref{se3} and Section~\ref{se4} is similar, see the diagram above and the diagram in the beginning of Section~\ref{se3}. Finally, in Section~\ref{se5}, we discuss the setting of sets of multiples corresponding to lattices and provide a counterexample to a lattice version of Theorem~\ref{equivalent}.

\section{Main objects}\label{se2}
\subsection{Number fields and ideals}\label{S:21}

Let $K$ be an algebraic number field with degree $d=[K:\Q]$, with the integer ring $\OK$. As in every Dedekind domain, all proper non-zero ideals in $\OK$ factor (uniquely, up to the order) into a product of prime ideals.
We will denote ideals in $\mathcal{O}_K$ by $\mathfrak{a},\mathfrak{b},\mathfrak{c}\dots$  We have
$$
\mathfrak{a}+\mathfrak{b}=\{a+b : a\in\mathfrak{a},b\in\mathfrak{b}\},\ \mathfrak{a}\mathfrak{b}=\{a_1b_1+\dots+a_kb_k : a_i\in \mathfrak{a},b_i\in\mathfrak{b},1\leq i\leq k\}.
$$
It is natural to set
\[
\gcd(\mathfrak{a},\mathfrak{b}):=\mathfrak{a}+\mathfrak{b} \text{ and }\lcm(\mathfrak{a},\mathfrak{b}):=\mathfrak{a}\cap\mathfrak{b}
\]
and speak of the \emph{greatest common divisor} and the \emph{least common multiple}, respectively. Note that the words ``greatest'' and ``least'' are a bit misleading here: $\gcd(\mathfrak{a},\mathfrak{b})$ is the \textbf{smallest} ideal containing both, $\mathfrak{a}$ and $\mathfrak{b}$, while $\lcm(\mathfrak{a},\mathfrak{b})$ is the \textbf{largest} ideal contained both in $\mathfrak{a}$ and $\mathfrak{b}$.

Proper ideals $\mathfrak{a},\mathfrak{b}$ are said to be \emph{coprime} whenever $\mathfrak{a}+\mathfrak{b}=\mathcal{O}_K$. Equivalently, $\mathfrak{a},\mathfrak{b}$ do not share factors: there are no non-trivial ideals $\mathfrak{a}',\mathfrak{b}',\mathfrak{c}$ such that $\mathfrak{a}=\mathfrak{c}\mathfrak{a}'$ and $\mathfrak{b}=\mathfrak{c}\mathfrak{b}'$. If $\mathfrak{a}$ and $\mathfrak{b}$ are coprime then $\lcm(\mathfrak{a},\mathfrak{b})=\mathfrak{a}\cap\mathfrak{b}=\mathfrak{a}\mathfrak{b}$.

The \emph{algebraic norm} of an ideal $\mathfrak{c}\neq \{0\}$ is defined as $N(\mathfrak{c}):=|\mathcal{O}_K/\mathfrak{c}|=[\mathcal{O}_K:\mathfrak{c}]$.

Finally, recall that there is a natural isomorphism from $\OK$ to a lattice in $\R^d$, called the \emph{Minkowski embedding} (see e.g.\ Chapter I, \S 5 in \cite{MR1697859}). We refer the reader to  \cite{MR0195803,MR1697859} for more background information on algebraic number theory.

\subsection{$\mathfrak{B}$-free integers in number fields}
Let $\mathfrak{B}$ be a collection of non-zero ideals in $\mathcal{O}_K$.
\begin{Def}
We say that
\begin{enumerate}[(i)]
\item ideal $\mathfrak{c}$ is \emph{$\mathfrak{B}$-free}, whenever $\mathfrak{c}\not\subseteq \mathfrak{b}$ for all $\mathfrak{b}\in\mathfrak{B}$ (equivalently, $\mathfrak{c}$ cannot be written as a product of $\mathfrak{b}$ with another ideal);
\item
integer $g\in\mathcal{O}_K$ is \emph{$\mathfrak{B}$-free} if the principal ideal $(g):=g\mathcal{O}_K$ is $\mathfrak{B}$-free.
\end{enumerate}
We denote the set of $\mathfrak{B}$-free integers in $\mathcal{O}_K$ by $\mathcal{F}_\mathfrak{B}$ and its complement by $\cM_\mathfrak{B}$.
\end{Def}
Since for any ideal $\mathfrak{b}\subseteq \OK$ and $g\in \mathcal{O}_K$ we have $g\not \in \mathfrak{b} \iff (g) \not\subseteq \mathfrak{b}$, it follows immediately that
\begin{equation*}\label{Z:1}
\mathcal{F}_\mathfrak{B}=\mathcal{O}_K\setminus \bigcup_{\mathfrak{b}\in\mathfrak{B}}\mathfrak{b}.
\end{equation*}
The characteristic function of $\mathcal{F}_\mathfrak{B}$ will be denoted by $\eta\in \{0,1\}^{\OK}$:
\begin{equation*}\label{Z:2}
\eta(g)=\begin{cases}
1& \text{if }g \text{ is }\mathfrak{B}\text{-free},\\
0& \text{otherwise}.
\end{cases}
\end{equation*}
For a finite subset $\mathfrak{S}\subset\mathfrak{B}$ define
\begin{equation*}
\ell_\mathfrak{S}=\bigcap_{\mathfrak{b}\in\mathfrak{S}}\mathfrak{b}=\lcm(\mathfrak{S})\quad\text{and}\quad
\cC_\mathfrak{S}=\{\mathfrak{b}+\ell_\mathfrak{S}:\mathfrak{b}\in\mathfrak{B}\}=\{\gcd(\mathfrak{b},\lcm(\mathfrak{S})):\mathfrak{b}\in\mathfrak{B}\}.
\end{equation*}
(cf.~\eqref{wym1} below).
We say that $\mathfrak{B}$ is \emph{primitive} whenever for any $\mathfrak{b},\mathfrak{b}'\in\mathfrak{B}$ if $\mathfrak{b}\subset\mathfrak{b}'$ then $\mathfrak{b}=\mathfrak{b}'$. For any $\mathfrak{B}$, there exists a unique primitive subset $\mathfrak{B}^{prim}\subset\mathfrak{B}$ such that $\mathcal{M}_{\mathfrak{B}^{prim}}=\mathcal{M}_{\mathfrak{B}}$:
\[
\mathfrak{B}^{prim}=\mathfrak{B}\setminus\{\mathfrak{b}\in\mathfrak{B} : \text{there exists }\mathfrak{b}'\in\mathfrak{B}\text{ such that }\mathfrak{b}\subsetneq\mathfrak{b}'\}.
\]
Without loss of generality we can therefore assume that $\mathfrak{B}$ is primitive. Then $\mathfrak{S}$ is a proper subset of $\cC_\mathfrak{S}$.
Let $\mathfrak{S}\subseteq \mathfrak{S}' \subset\mathfrak{B}$. Then
\begin{equation*}\label{eq:A_S-inclusions}
\mathfrak{S}\subseteq\mathfrak{S}'\subseteq\cC_{\mathfrak{S}'}\subseteq\cM_{\cC_\mathfrak{S}}\text{\quad so that\quad}
\cM_\mathfrak{S}\subseteq\cM_{\mathfrak{S}'}\subseteq\cM_{\cC_{\mathfrak{S}'}}\subseteq\cM_{\cC_\mathfrak{S}}.
\end{equation*}
A finite set $\mathfrak{S}\subset\mathfrak{B}$ is \emph{saturated}, if $\cC_\mathfrak{S}\cap\mathfrak{B}=\mathfrak{S}$. In other words, the only elements of $\mathfrak{B}$ that divide $\lcm(\mathfrak{S})$ are precisely the elements of $\mathfrak{S}$.
 Given any $\mathfrak{S}\subset \mathfrak{B}$, the set $\mathfrak{S}^{sat}:=\cC_\mathfrak{S}\cap\mathfrak{B}$ is saturated. Moreover, if $\mathfrak{S}_1\subset \mathfrak{S}_2\subset\dots \nearrow \mathfrak{B}$ is a filtration, then, for some subsequence $(n_k)$,
\[
\mathfrak{S}_{n_1}^{sat}\subset \mathfrak{S}_{n_2}^{sat} \dots\subset\nearrow \mathfrak{B}
\]
is a saturated filtration of $\mathfrak{B}$, i.e.\ a filtration consisting of saturated sets (see~\cite{DKK}, Section~3 for the details in dimension one).

\paragraph{Dynamical system outputting $\mathfrak{B}$-free integers}\label{szift}
Given an abelian group $\mathbb{G}$ and a finite alphabet $\mathcal{A}$, there is a natural action on $\mathcal{A}^\mathbb{G}$ by $\mathbb{G}$ by commuting translations: $S_g({(x_{g'})}_{h\in \mathbb{G}})={(x_{g'+g})}_{g'\in \mathbb{G}}\text{ for }g,g'\in \mathbb{G}.$
In particular, on $\{0,1\}^{\OK}$, we have
\begin{equation*}\label{eq:translations}
S_g({(x_{g'})}_{g'\in \mathcal{O}_K})={(x_{g'+g})}_{g'\in \mathcal{O}_K}, g\in\OK.
\end{equation*}
To introduce a metric on $\{0,1\}^{\OK}$, recall that $(F_n)_{n\geq 1}\subseteq \OK$ is a \emph{F\o lner sequence} if for any $g\in \OK$,
\[
\lim_{n\to \infty}\frac{|(F_n+g)\cap F_n|}{|F_n|}=1.
\]
If $\bigcup\limits_{n\geq 1}F_n=\OK$ and~$F_n\subseteq F_{n+1}$ for any $n\geq 1$, then $(F_n)_{n\geq 1}$ is called \emph{nested} \cite{MR64062, MR79220}.

Fix a nested F\o lner sequence $(F_n)_{n\geq1}\subseteq\OK$ (we can assume $F_1=\{1\}$) and consider the metric given by the formula
\begin{equation}\label{metryka}
d(x,y)=\begin{cases}
&1, \text{ if } x_1\neq y_1, \\
&2^{-\max\{n\geq1 ;\ x_g=y_g \text{ for any } g\in F_n \} } \text{ otherwise.}
\end{cases}
\end{equation}
The above metric induces the product topology on $\{0,1\}^{\OK}$.
If $X\subseteq \{0,1\}^{\OK}$ is closed and ${(S_g)}_{g\in\OK}$-invariant, we say that $X$ is a \emph{subshift}. We will denote by $X_\eta\subseteq \{0,1\}^{\mathcal{O}_K}$ the smallest subshift containing $\eta$. A subshift $X$ is called \emph{minimal} if for any $x\in X$ its orbit $\{S_gx: g\in\OK\}$ is dense in $X$. We say that $x\in \{0,1\}^{\OK}$ is \emph{minimal}, if its orbit closure is a minimal subshift. Let us recall one more notion, in a sense complementary to the minimality in the case of subshifts. We say that a pair $(x,y)\in X\times X$ is \emph{proximal}, if \begin{equation}\label{prox}
\liminf_{g\to\infty}d(S_gx,S_gy)=0.
\end{equation}
By $g_n\to\infty$ as $n\to\infty$ we mean that for any finite $A\subseteq\OK$, there exists $n_0\geq1$ such that $g_n\not\in A$ for any $n>n_0$. If any pair of points in~$X$ is proximal, then subshift is called {\em proximal}. Usually the opposite notion to proximality is distality (a subshift $X$ is called \emph{distal} if \eqref{prox} is not satisfied for any $x\neq y\in X$). But for any subshift $X$ its distality implies its equicontinuity, so the finiteness of $X$, see \cite[Theorem 2]{MR101283}.

Consider
\begin{equation*}
H_\mathfrak{B}=\prod_{\mathfrak{b}\in\mathfrak{B}}\mathcal{O}_K/{\mathfrak{b}}
\end{equation*}
with the coordinatewise addition. It is the product of finite groups $\OK/{\mathfrak{b}}$ and the Haar measure $\PP_\mathfrak{B}$ on $H_\mathfrak{B}$ is the product of the corresponding counting measures. Moreover, there is a natural $\mathcal{O}_K$-action on $H_\mathfrak{B}$ by translations:
\begin{equation*}\label{eq:rota}
T_g((h_{\mathfrak{b}})_{\mathfrak{b}\in\mathfrak{B}})=(h_\mathfrak{b}+g)_{\mathfrak{b}\in\mathfrak{B}}, g\in\OK.
\end{equation*}
Let
\begin{equation}\label{eq:grupa}
H=\overline{\{T_g(\underline{0}):g\in\OK\}},
\end{equation}
where $\underline{0}$ is the neutral element of $H_\mathfrak{B}$ and let $\PP$ be the Haar measure on $H$.

Let $\varphi\colon H\to \{0,1\}^{\mathcal{O}_K}$ be defined as
\begin{equation*}\label{eq:wzor}
\varphi(h)(g)=\begin{cases}
1,& \text{ if }h_{\mathfrak{b}}+g\not\equiv 0\bmod \mathfrak{b} \text{ for each }\mathfrak{b}\in\mathfrak{B},\\
0,&\text{ otherwise},
\end{cases}
\end{equation*}
where $h=(h_\mathfrak{b})_{\mathfrak{b}\in\mathfrak{B}}$. Notice that $\varphi(\underline{0})=\eta=\raz_{\mathcal{F}_\mathfrak{B}}$. Moreover, $\varphi={(\raz_C\circ T_g)}_{g\in\OK}$, where
\begin{equation*}\label{eq:ce}
C=\{g\in H : h_\mathfrak{b}\not\equiv 0 \bmod \mathfrak{b}\text{ for each }\mathfrak{b}\in\mathfrak{B}\}.
\end{equation*}
In other words, $\varphi$ is the coding of orbits of points under ${(T_g)}_{g\in\OK}$ with respect to the partition $\{C,G\setminus C\}$ of $G$.

Finally, let $\nu_\eta:=\varphi_\ast(\PP)$ be the pushforward of $\PP$ under $\varphi$.
We will call $\nu_\eta$ the \emph{Mirsky measure}.
\begin{Remark}
All above objects are definined for $\sB$-free numbers by putting $\OK=\Z$ and $\mathfrak{b}=b\Z$ for any $b\in\sB\subset\N$.
\end{Remark}
\begin{Remark}
In the case of $\{p^k : p\in\mathcal{P}\}$-free numbers, in particular in the square-free case, $\nu_\eta$  was considered by Mirsky~\cite{MR21566,MR28334} (cf.\ also~\cite{MR30561}) who studied the frequencies of blocks on $\eta$.
\end{Remark}
\subsection{Toeplitz sequences}
\subsubsection{Dimension one}
Let $\mathcal{A}$ be a finite alphabet. We say that $\bm{x}=(\bm{x}_n)_{n\in\Z}\in\mathcal{A}^\Z$ is a \emph{Toeplitz sequence}~\cite{MR255766}, whenever for any $n\in\Z$ there exists $s_n\in\N$ such that $x_{n+k\cdot s_n}=x_n$ for every $k\in\Z$. In other words, any symbol appears on $\bm{x}$ with some period. The orbit closure of any Toeplitz sequence is called a \emph{Toeplitz subshift}. Given $\bm{x}\in \mathcal{A}^\Z$, let
\[
\text{Per}(\bm{x},s,a):=\{n\in\mathbb{Z} : \bm{x}|_{n+s\Z}=a\},\text{ where }a\in\mathcal{A},\ s\in\mathbb{N}
\]
and
\[
\text{Per}(\bm{x},s)=\bigcup_{a\in\mathcal{A}}\text{Per}(\bm{x},s,a).
\]
It is easy to see that $\bm{x}\in\mathcal{A}^\Z$ is a Toeplitz sequence if and only if there exists a sequence $(p_n)_{n\geq 1}$ such that
\begin{equation*}\label{dodane}
p_n\divides p_{n+1} \text{ and }\Z=\bigcup_{n\geq 1}\text{Per}(\bm{x},p_n).\footnotemark
\end{equation*}
\footnotetext{\label{stopa}Indeed, one can simply take any $p_1$ such that $0\in \text{Per}(\bm{x},p_1)$ and then proceed inductively. Suppose that we have chosen $p_1 \divides p_2 \divides\dots \divides p_{2n+1}$ such that $[-n,n]\subset \bigcup_{1\leq k\leq 2n+1}\text{Per}(\bm{x},p_k)$. There exist $p_{2n+2},p_{2n+3}$ such that $-n-1\in \text{Per}(\bm{x},p_{2n+2})$, $n+1\in \text{Per}(\bm{x},p_{2n+3})$. Replacing $p_{2n+2}$ with $\lcm(p_{2n+1},p_{2n+2})$ and then $p_{2n+3}$ with $\lcm(p_{2n+2},p_{2n+3})$, we obtain $p_{2n+1}\divides p_{2n+2}\divides p_{2n+3}$ and $[-n-1,n+1]\subset \bigcup_{1\leq k\leq 2n+3}\text{Per}(\bm{x},p_k)$.}

Finally, let $?$ be a fixed symbol that is not in $\cA$, called a \emph{hole}. For any $s\geq1$ and a Toeplitz subshift $X$ we define the \emph{skeleton map at scale} $s$ by $M_{s}\colon X \rightarrow(\cA \cup ?)^{\mathbb{Z}}$, where
\[
\left(M_{s}(\bm{y})\right)_n=
\begin{cases}
\bm{y}_n,\ &n \in \operatorname{Per}(\bm{y},s),\\
?,\ & \text{ otherwise}
\end{cases}
\]
for any $\bm{y}\in X$. Moreover, if $\bm{y}\in X_{\bm{x}}$, where $\bm{x}$ is Toeplitz then
\begin{equation}\label{downar}
M_{s}(\bm{y})=\sigma^{j} M_{s}(\bm{x})=M_{s}(\sigma^{j}\bm{x})
\end{equation}
for some $0\leq j\leq s-1$. Indeed, since $\bm{y}\in  X_{\bm{x}}$, it follows immediately that $\Per(\bm{y},s)\supset \Per(\bm{x},s)-j$ for some $1\leq j\leq s-1$. By the minimality of $X_{\bm{x}}$, we have $\bm{x}\in X_{\bm{y}}$ and we can reverse the roles of $\bm{x}$ and $\bm{y}$ to obtain $\Per(\bm{y},s)= \Per(\bm{x},s)-j$. Moreover,~\eqref{downar} also holds.

\paragraph{Essential periods}
Downarowicz, in his survey~\cite{MR2180227}, defines essential periods for a given Toeplitz sequence $\bm{x}\in\mathcal{A}^\Z$ in the following way:
\begin{Def}[{\cite[Definition 7.3]{MR2180227}}]\label{pierwsza}
Number $s\in \mathbb{N}$ is said to be an \textit{essential period} of $\bm{x}\in\mathcal{A}^\Z$ if $\text{Per}(\bm{x},s)\neq\emptyset$ and
\begin{equation}\label{one}
q<s \implies \text{Per}(\bm{x},s)\neq \text{Per}(\bm{x},q).
\end{equation}
\end{Def}
Clearly,~\eqref{one} is equivalent to
\begin{equation}\label{one1}
\text{Per}(\bm{x},s)= \text{Per}(\bm{x},q) \implies q\geq s.
\end{equation}
Downarowicz gives a reference to a paper by Williams~\cite{MR756807}. However, she formulates the definition of an essential period differently.
\begin{Def}[\cite{MR756807}]\label{druga}
Number $s\in \mathbb{N}$ is said to be an \textit{essential period }of $\bm{x}\in\mathcal{A}^\Z$ if $\text{Per}(\bm{x},s)\neq\emptyset$ and
\begin{equation}\label{dwa}
\left(\Per(\bm{x},s,a)=\Per(\bm{x},s,a)-q \text{ for all }a\in\mathcal{A}\right) \implies s\divides q.
\end{equation}
\end{Def}

We claim that the two above notions mean the same.
\begin{Prop}\label{twierd}
Definitions~\ref{pierwsza} and~\ref{druga} are equivalent. Moreover, in formula~\eqref{one1}, condition $q\geq s$ can be replaced by $s\divides q$.
\end{Prop}
This can be proven directly (and it is not difficult), however, let us first prove the following result which will be useful when we pass to the multidimensional setting.
\begin{Prop}\label{equiv_cond}
Fix $\bm{x}\in\mathcal{A}^\Z$ and $s,q\in \mathbb{N}$. The following conditions are equivalent:
\begin{enumerate}
\item[(a)]
$\Per(\bm{x},s,a)=\Per(\bm{x},s,a)-q \text{ for all }a\in\mathcal{A}$,
\item[(b)]
$\Per(\bm{x},s,a)=\Per(\bm{x},s,a)-q\mathbb{Z} \text{ for all }a\in\mathcal{A}$,
\item[(c)]
$\Per(\bm{x},s,a)=\Per(\bm{x},\gcd(s,q),a)\text{ for all }a\in \mathcal{A}$,
\item[(d)]
$\Per(\bm{x},s,a)\subset \Per(\bm{x},q,a)$ for all $a\in \mathcal{A}$,
\item[(e)]
$\Per(\bm{x},s)\subset \Per(\bm{x},q)$.
\end{enumerate}
\end{Prop}
\begin{proof}
It is immediate that (a) and (b) are equivalent (to go from (a) to (b) we just apply (a) repeatedly and to go from (b) to (a) we just use that $\Per(\bm{x},s,a)-q$ is contained in $\Per(\bm{x},s,a)-q\mathbb{Z}$ and both sides of (a) are unions of the same number of arithmetic progressions of the same difference $s$). The equivalence of (d) and (e) follows from the fact that $\Per(\bm{x},s)=\bigcup_{a\in\mathcal{A}}\Per(\bm{x},s,a)$, $\Per(\bm{x},q)=\bigcup_{a'\in\mathcal{A}}\Per(\bm{x},q,a')$, where
\[
\Per(\bm{x},s,a)\cap\Per(\bm{x},q,a')=\emptyset \text{ whenever }a\neq a'.
\]
If this intersection is nonempty then clearly $a=\bm{x}_n=a'$.

We prove now that (b) implies (c). Assume (b) and take $n\in \Per(\bm{x},s,a)$. We have $n+s\Z\subset \Per(\bm{x},s,a)\subset \Per(\bm{x},q,a)$. Therefore, for each $n\in \Per(\bm{x},s,a)$, we have $x|_{n+s\Z+q\Z}\equiv a$. However, this means that $n\in \Per(\bm{x},\gcd(s,q),a)$ since $s\Z+q\Z=\gcd(s,q)\Z$. This gives $\Per(\bm{x},s,a)\subset \Per(\bm{x},\gcd(s,q),a)$ for each $a\in \mathcal{A}$. The opposite inclusion always takes place as $\gcd(s,q)\divides s$. Notice also that (d) follows from (c) immediately as $\gcd(s,q)\divides q$.

It remains to show that (d) implies (b). Take $n\in\Per(\bm{x},s,a)$. Then $n+s\Z\subset \Per(\bm{x},s,a)\subset \Per(\bm{x},q,a)$. Therefore, $n+s\Z+q\Z\subset \Per(\bm{x},q,a)$. In particular, $x|_{n+q\Z+s\Z}\equiv a$, which yields $n+q\Z\subset \Per(\bm{x},s,a)$. This completes the proof.
\end{proof}
\begin{Remark}\label{R1}
As an immediate consequence, we obtain that condition~\eqref{dwa} from Definition~\ref{druga} is equivalent to
\begin{equation}\label{trzeci}
\Per(\bm{x},s)\subset \Per(\bm{x},q) \implies s\divides q.
\end{equation}
\end{Remark}
\begin{Remark}\label{R2}
We claim that condition~\eqref{one} from Definition~\ref{pierwsza} is equivalent to
\begin{equation}\label{czwarty}
\Per(\bm{x},s)= \Per(\bm{x},q) \implies s\divides q.
\end{equation}
In fact,~\eqref{one1},~\eqref{trzeci} and~\eqref{czwarty} are all equivalent. Indeed,~\eqref{trzeci} implies~\eqref{czwarty} and~\eqref{czwarty} implies~\eqref{one1}. It remains to show that~\eqref{one1} implies~\eqref{trzeci}. Suppose that $\Per(\bm{x},s)\subset \Per(\bm{x},q)$ and that~\eqref{one1} holds. It follows by Proposition~\ref{equiv_cond} ((e)$\implies$(c)) that $\Per(\bm{x},s)=\Per(\bm{x},\gcd(s,q))$, so we can apply~\eqref{one1} to conclude that $\gcd(s,q)\geq s$. The latter condition is however equivalent to $s\divides q$.
\end{Remark}

\begin{proof}[Proof of Proposition~\ref{twierd}]
We combine Remark~\ref{R1} and Remark~\ref{R2}. To complete the proof, it remains to use the implication (e)$\implies$(c) from Proposition~\ref{equiv_cond}. Indeed, if $\Per(\bm{x},s)\subset \Per(\bm{x},q)$ then $\Per(\bm{x},s)=\Per(\bm{x},\gcd(s,q))$. Therefore, if~\eqref{czwarty} holds for $s$ then $s\divides\gcd(s,q)$, which is equivalent to $s=\gcd(s,q)$, which, in turn, is the same as $s\divides q$.
\end{proof}
\begin{Remark}\label{NOWA}
Notice that $s\divides q$ is equivalent to $q\Z\subset s\Z$. Moreover, for any $k\in\N$, $k\Z\subset \Z$ is a subgroup and an ideal of $\Z$.
\end{Remark}

\paragraph{Periodic structure}
A \emph{periodic structure} \cite{MR756807} of a Toeplitz sequence $\bm{x}\in\mathcal{A}^\Z$ is any sequence $(p_k)_{k\in \N}$ of essential periods such that $p_k\divides p_{k+1}$ for each $k\in\N$ and
$\bigcup_{k\in\N} \operatorname{Per}(\bm{x},p_k) = \Z$. Every Toeplitz sequence has a periodic structure.

\subsubsection{Higher dimension (abelian discrete, finitely generated groups)}
Let $\mathbb{G}$ be a discrete, finitely generated group and let $\Gamma\subseteq \G$ be a subgroup of finite index. We will be mostly interested in case $\G=\Z^d$ and therefore we restrict ourselves to abelian groups. For $\bm{x}=(\bm{x}_g)_{g\in\G}\in\mathcal{A}^\G$ consider
\[\Per(\bm{x},\Gamma,a)=\{g\in\G; \ \bm{x}_{g+\gamma}=a \text{ for any } \gamma\in\Gamma\}, \quad a\in\mathcal{A},\]
\[\Per(\bm{x},\Gamma)=\bigcup_{a\in\mathcal{A}}\Per(\bm{x},\Gamma,a).\]
If $\Per(\bm{x},\Gamma)\neq\emptyset$, then we say that $\Gamma$ is a \emph{group of periods of }$\bm{x}$.
We say that $\bm{x}$ is a~\emph{$\G$-Toeplitz array}, if for all $g\in\G$ there exists a subgroup $\Gamma\subseteq\G$ of finite index such that $g\in\Per(\bm{x},\Gamma)$. The orbit closure of a $\G$-Toeplitz array endowed with multidimensional shifts $(S_g)_{g\in \G}$ is called a~\emph{$\G$-Toeplitz system}.
\begin{Th}[{\cite[Proposition 5]{MR2427048}},  cf.~\eqref{trzeci} and Remark~\ref{NOWA}]
\label{Cor-Pet}
Let $\bm{x}\in\mathcal{A}^{\G}$. The following are equivalent:
\begin{enumerate}[(1)]
\item $\bm{x}$ is a $\G$-Toeplitz array,
\item there exists a sequence $(\Gamma_n)_{n\geq1}$ of groups with finite indices such that $\Gamma_{n+1}\subset\Gamma_n$ for any $n\geq1$ and $\G=\bigcup_{n\geq1}\Per(\bm{x},\Gamma_n)$.
\end{enumerate}
\end{Th}

If $\G=\Z^d$ then conditions (1) and (2) from the above theorem are equivalent to the following one:
\begin{enumerate}
\item[(3)]
there exists a sequence $(p_n)_{n\geq 1}\subset \N$ such that $p_n\divides p_{n+1}$ and $\G=\bigcup_{n\geq1}\text{Per}(\bm{x},p_n\Z^d)$.
\end{enumerate}
Indeed, it suffices to see that for any subgroup $\Gamma_n\subset \Z^d$ of finite index, there exists $p_n\in\N$ with $\Gamma_n\supset p_n\Z^d$ and proceed inductively, as in footnote~\ref{stopa} (for a detailed proof, see {\cite[Proposition 14]{MR2220753}}).
\begin{Remark}
Take $\G=\OK$ (as in Section \ref{szift}). Since there is a group isomorphism between the additive structures of $\OK$ and $\Z^d$ and it preserves the index of a subgroup, it follows immediately that for an $\OK$-Toeplitz array $\bm{x}\in\mathcal{A}^{\OK}$, there is a corresponding $\Z^{d}$-Toeplitz array $\bm{y}\in\mathcal{A}^{\Z^{d}}$. By condition (3) above, there exist a sequence of groups of periods $(p_n\Z^{d})_{n\in\N}$ for $\bm{y}$ such that $p_n\divides p_{n+1}$ and $\bigcup_{n\geq1}\text{Per}(\bm{y},p_n\Z^{d})=\Z^{d}$. Any such group of periods $p_n\Z^{d}$ for $\bm{y}$ corresponds to a group of periods for $\bm{x}$ which is a principal ideal.
\end{Remark}


As in dimension one, let $?$ be a fixed symbol that is not in $\cA$, called a \emph{hole}. For any subgroup $\Gamma\subseteq \G$ of finite index and a $\G$-Toeplitz system $X$, we define the \emph{skeleton map at scale} $\Gamma$ by $M_\Gamma\colon X\rightarrow (\cA\cup ?)^{\G}$, where
\[
\left(M_{\Gamma}(\bm{y})\right)_n=
\begin{cases}
\bm{y}_n,\ &n \in \operatorname{Per}(\bm{y},{\Gamma}),\\
?,\ & \text{ otherwise}
\end{cases}
\]
for any $\bm{y}\in X$. Moreover, if $\bm{y}\in X_{\bm{x}}$, where $\bm{x}$ is a $\G$-Toeplitz array then
\begin{equation}\label{downar_mul}
M_{\Gamma}(\bm{y})=\sigma^{j} M_{\Gamma}(\bm{x})=M_{\Gamma}(\sigma^{j}\bm{x})
\end{equation}
for some $j\in \G$. The proof of \eqref{downar_mul} is the same as the proof of \eqref{downar} in the one-dimensional case.

\paragraph{Essential periods}

There are two definitions of essential periods in this setting present in the literature. The goal of this section is to show that they are equivalent.


Cortez, in her paper \cite{MR2220753}, defines the notion for $\G=\Z^d$ in the following way.
\begin{Def}[{\cite[Definition 15]{MR2220753}}]\label{essential_C}
A group $Z\subset\Z^d$ of periods of $\bm{x}\in \mathcal{A}^{\Z^d}$ is called a \emph{group generated by essential periods of} $\bm{x}$ if $\Per(\bm{x},Z)\subseteq\Per(\bm{x},Z')$ implies that $Z'\subseteq Z$.
\end{Def}
Cortez and Petite~\cite{MR2427048} define that notion for more general $\G$-Toeplitz arrays differently (in their paper $\G$ is a discrete finitely generated group). We recall it in the abelian setting here.
\begin{Def}[{\cite[Definition 4]{MR2427048}}]\label{essential_CP}
A syndetic group $\Gamma\subset\G$ is called an \emph{essential group of periods of} $\bm{x}\in\mathcal{A}^\G$ if $\Per(\bm{x},\Gamma,a)\subseteq\Per(S_g\bm{x},\Gamma,a)$ for every $a\in\mathcal{A}$
implies that $g\in\Gamma$.
\end{Def}
\begin{Remark}
Notice that since $\Per(\bm{x},\Gamma,a)$ is a finite union of sets in the form $\Gamma+g$, $g\in\G$, and for any $a\in\mathcal{A}$ one has $\Per(S_g\bm{x},\Gamma,a)=\Per(\bm{x},\Gamma,a)-g$, the inclusion in Definition \ref{essential_CP} implies $\Per(\bm{x},\Gamma,a)=\Per(\bm{x},\Gamma,a)-g$. Hence it is clear that for $\G=\Z$ the two definitions are equivalent: $\bm{x}\in\mathcal{A}^\Z$, $s\geq 1$ is an essential period if and only if $s\Z$ is an essential group of periods.
\end{Remark}
\begin{Remark}\label{No1}
Cortez and Petite remark that any essential group of periods of $\bm{x}\in\mathcal{A}^{\Z^d}$ is a group generated by essential periods of $\bm{x}$ (\cite[Remark 2]{MR2427048}). In fact, Definition~\ref{essential_C} can be easily extended to a general finitely generated group $\G$. Lemma 6 in \cite{MR2427048} shows that any essential group of periods of $\bm{x}\in\mathcal{A}^\G$ is a group generated by essential periods of $\bm{x}$. We will show that also the opposite implication is true, see Corollary~\ref{essential_multi} below. In order to do this, we will use a multidimensional version of Proposition \ref{equiv_cond}.
\end{Remark}
\begin{Prop}\label{multi_equiv_cond}
Fix $g\in\G$ and let $\Gamma\subset\G$ be a subgroup. Let $\Gamma'=\langle\Gamma,g\rangle$. The following conditions are equivalent:
\begin{enumerate}
\item[(a)]
$\Per(\bm{x},\Gamma,a)=\Per(\bm{x},\Gamma,a)-g \text{ for all }a\in\mathcal{A}$,
\item[(b)]
$\Per(\bm{x},\Gamma,a)=\Per(\bm{x},\Gamma,a)-\Gamma' \text{ for all }a\in\mathcal{A}$,
\item[(c)]
$\Per(\bm{x},\Gamma,a)=\Per(\bm{x},\Gamma',a)\text{ for all }a\in \mathcal{A}$,
\item[(d)]
$\Per(\bm{x},\Gamma,a)\subset \Per(\bm{x},\Gamma',a)$ for all $a\in \mathcal{A}$,
\item[(e)]
$\Per(\bm{x},\Gamma)\subset \Per(\bm{x},\Gamma')$.
\end{enumerate}
\end{Prop}
\begin{proof}
Since $\langle\Gamma,g\rangle=\{a\gamma+b g:  a,b\in\Z, \gamma\in\Gamma\}$, to show that (a) implies (b), we use $\Gamma$-invariance of $\Per(\bm{x},\Gamma,a)$ and just apply (a) repeatedly.

To go from (b) to (a) we just use that $\Per(\bm{x},\Gamma,a)-g\subseteq\Per(\bm{x},\Gamma,a)-\Gamma'$ and both sides of (a) are unions of the same number of sets in the form $\Gamma+g'$, $g'\in\G$.

The equivalence of (d) and (e) follows from the fact that $\Per(\bm{x},\Gamma)=\bigcup_{a\in\mathcal{A}}\Per(\bm{x},\Gamma,a)$, $\Per(\bm{x},\Gamma')=\bigcup_{a'\in\mathcal{A}}\Per(\bm{x},\Gamma',a')$, where
\[
\Per(\bm{x},\Gamma,a)\cap\Per(\bm{x},\Gamma',a')=\emptyset \text{ whenever }a\neq a'
\]
(if this intersection is nonempty then clearly $a=\bm{x}_n=a'$).

We prove now that (b) implies (c). Assume (b) and take $\gamma\in \Per(\bm{x},\Gamma,a)$. Then $\bm{x}|_{\gamma+\Gamma'}\equiv a$. So $\gamma\in\Per(\bm{x},\Gamma',a)$. This gives $\Per(\bm{x},\Gamma,a)\subset \Per(\bm{x},\Gamma',a)$ for each $a\in \mathcal{A}$. The opposite inclusion always takes place as $\Gamma\subset\Gamma'$. Notice also that (d) follows from (c) immediately.

It remains to show that (d) implies (a). Take $\gamma\in\Per(\bm{x},\Gamma,a)$. Then by (d), $\gamma-g\in\gamma+\Gamma'\subset\Per(\bm{x},\Gamma',a)$. In particular, $\bm{x}|_{\gamma-g+\Gamma'}\equiv a$, which yields $\bm{x}|_{\gamma-g+\Gamma}\equiv a$, so $\Per(\bm{x},\Gamma,a)-g\subset\Per(\bm{x},\Gamma,a)$. Since both sets are unions of the same number of sets in the form $\Gamma+g'$, $g'\in\G$, they are equal. This completes the proof.
\end{proof}

\begin{Cor}[cf.\ Remark~\ref{No1}]\label{essential_multi}
For any $\bm{x}\in\mathcal{A}^\G$, the family of corresponding groups generated by essential periods is the same as the family of corresponding essential groups of periods.
\end{Cor}
\begin{proof}
Suppose that $\Gamma$ is a group generated by essential periods of $\bm{x}$ and $\Per(\bm{x},\Gamma,a)\subset\Per(\bm{x},\Gamma,a)-g$ for every $a\in\mathcal{A}$ and for some $g\in\G$. Let $\Gamma'=\langle\Gamma,g\rangle$. By Proposition \ref{multi_equiv_cond} (a)$\Rightarrow$(e), $\Per(\bm{x},\Gamma)\subset\Per(\bm{x},\Gamma')$. Since $\Gamma$ is a group generated by essential periods of $\bm{x}$, $\Gamma'\subset\Gamma$. So $g\in\Gamma$. Hence $\Gamma$ is an essential group of periods.

Suppose now that $\Gamma$ is an essential group of periods of $\bm{x}$ and $\Per(\bm{x},\Gamma)\subset\Per(\bm{x},\Gamma')$ for some $\Gamma'$ group of periods of $\bm{x}$. Let $g\in\Gamma'$, $a\in\mathcal{A}$ and $\gamma\in\Per(\bm{x},\Gamma,a)$. Then $\bm{x}|_{\gamma+\Gamma+g}\equiv a$, so $\gamma\in\Per(\bm{x},\Gamma,a)-g$. Thus, $\Per(\bm{x},\Gamma,a)\subset\Per(\bm{x},\Gamma,a)-g$. Since $\Gamma$ is an essential group of periods of $\bm{x}$, $g\in\Gamma$. Therefore $\Gamma'\subset\Gamma$, whence $\Gamma$ is a group generated by essential periods of $\bm{x}$.
\end{proof}

Recall also that essential groups of periods exist. In fact, the next lemma tells us that we have even more.
\begin{Lemma}[{\cite[Lemma 7]{MR2427048}}]
Let $\bm{x} \in \mathcal{A}^{\G}$. If $\Gamma \subseteq \G$ is a group of periods of $\bm{x}$, then there exists $K \subseteq \G$ an essential group of periods of $\bm{x}$ such that \[\operatorname{Per}(\bm{x}, \Gamma) \subseteq \operatorname{Per}(\bm{x}, K).\]
\end{Lemma}


\paragraph{Periodic structure and maximal equicontinuous factor}
\begin{Th}[{\cite[Corollary 6]{MR2427048}}]\label{tw_essential}
Let $\bm{x} \in \mathcal{A}^{\G}$ be a $\G$-Toeplitz array. Then there exists a sequence $\left(\Gamma_{n}\right)_{n \geq 1}$ of essential groups of periods of $\bm{x}$ such that
\[
\Gamma_{n+1} \subseteq \Gamma_{n}\text{ and }\bigcup_{n\geq1} \operatorname{Per}\left(\bm{x}, \Gamma_{n}\right)=\G.
\]
\end{Th}
Sequence $(\Gamma_n)$ from the theorem above is called a \emph{periodic structure} of $\bm{x}$. It can be used to describe the maximal equicontinuous factor of a $\G$-Toeplitz system.
\begin{Th}[{\cite[Proposition 7]{MR2427048}}]
Let $\bm{x} \in \mathcal{A}^{\G}$ be a $\G$-Toeplitz array. If $\left(\Gamma_{n}\right)_{n\geq1}$ is a periodic structure of $\bm{x}$, then $\overleftarrow{G}=\varprojlim_{ n}\G / \Gamma_{n}$ with addition of $(e_\G,e_\G,\ldots)$ is the maximal equicontinuous factor of $\left(X_{\bm{x}}, (S_g)_{g\in\G}\right)$.
\end{Th}
We will also need the following related technical result (used in~\cite{MR2427048} to prove Proposition 7 therein).
\begin{Th}[{\cite[Proposition 6, Remark 2]{MR2427048}}]\label{structure}
Let $\bm{x} \in \mathcal{A}^{\G}$ be a $\G$-Toeplitz array and $\bm{y} \in X_{\bm{x}}$. If $\Gamma_{\bm{y}} \subseteq \G$ is an essential group of periods of $\bm{y}$, then
\[
C_{\Gamma_{\bm{y}}}(\bm{y}):=\left\{\bm{x}' \in X_{\bm{x}} : \operatorname{Per}\left(\bm{x}', \Gamma_{\bm{y}}, a\right)=\operatorname{Per}(\bm{y}, \Gamma_{\bm{y}}, a) \text{ for all } a \in \cA\right\}=\overline{\{S_g\bm{y};\ g\in\Gamma_{\bm{y}}\}}
\]
and $\left\{S_w (C_{\Gamma_{\bm{y}}}(\bm{y}))\right\}_{w \in \G / \Gamma_{\bm{y}}}$ is a closed partition of $X_{\bm{x}}$.
\end{Th}

\section{Minimality and periodic structure of $\sB$-free systems}\label{se3}
\subsection{Background and main results}
Recall that any $\sB$-free system is essentially minimal, i.e.\ it has exactly one minimal subset. In fact, we have the following.
\begin{Th}[{\cite[Theorem A]{MR3803141}},  ``baby version'' of Theorem~\ref{minimal2} in dimension one]\label{minimal_basic}
For any $\mathscr{B}\subset \N$, the subshift $(X_\eta,\sigma)$ has a unique minimal subset, denoted by $M$. Moreover, $M$ is the orbit closure of a Toeplitz sequence.
\end{Th}
Recall the notion of densities for subsets of integers. Given $A\subset \Z$, let $\bm{\delta}(A)$ stand for the logarithmic density and $\bm{d}(A)$ stand for the natural density of $A$, provided that these densities exist:
\[
\bm{\delta}(A)=\lim_{N\to \infty}\frac{1}{\log N}\sum_{1\leq n\leq N}\frac{\bm{1}_A(n)}{n} \text{ and }\bm{d}(A)=\lim_{N\to \infty}\frac{1}{N}\sum_{1\leq n\leq N}\bm{1}_A(n).
\]
Recall that for the sets of multiples the logaritmic density always exists, by the result of Davenport and Erd\"os~\cite{MR43835}. We say that $\mathscr{B}$ is \emph{Behrend}, whenever $\bm{\delta}(\mathcal{M}_\mathscr{B})=1$ (which is equivalent to $\bm{d}(\mathcal{M}_\mathscr{B})=1$). Recall also that a set $\mathscr{B}$ is called \emph{taut} when
\[
\bm{\delta}(\mathcal{M}_{\mathscr{B}})>\bm{\delta}(\mathcal{M}_{\mathscr{B}\setminus \{b\}}) \text{ for any }b\in\mathscr{B}.
\]
It follows by~\cite[Corollary 0.19]{MR1414678} that $\mathscr{B}$ is taut precisely if it is primitive and
\[
d\cC\not\subseteq \mathscr{B} \text{ for any }d\in\N \text{ and any Behrend set }\cC\subset\N.
\]
Recall that $\mathscr{B}$ is primitive if $b\ndivides b'$ for any $b\neq b'\in\mathscr{B}$.


Kasjan, Keller and Lema\'nczyk described the minimality of $\sB$-free systems in~\cite{MR3947636}. Before we formulate their result, let us define some auxiliary objects. Given $\sB$, let
\begin{align}
\mathcal{C}_{S}&:=\{\gcd(b,\lcm(S)):  b\in\sB\} \text{ for a finite subset }S\subset \sB,\label{wym1}\\
\mathcal{C}_{\infty}&:=\left\{n \in \mathbb{N}; \ \forall_{\text{finite }S\subset \sB}\ \exists_{\text{finite }S^{\prime};\ S \subseteq S^{\prime} \subset \mathscr{B}}\ n \in \mathcal{C}_{S^{\prime}} \setminus S^{\prime}\right\}.\nonumber
\end{align}
\begin{Th}[{\cite[Theorem B]{MR3947636}}, almost complete version of Theorem~\ref{equivalent} in dimension one]\label{toeplitz}
Suppose that $\sB$ is primitive. Let $W=\{h\in H: \ h_b\neq0\text{ for all }b\in\sB\}$, where $H$ is defined by \eqref{eq:grupa} for $\OK=\Z$ and $\mathfrak{b}=b\Z$, $b\in\sB$, i.e.
\[
H=\overline{\{(n,n,n,\dots) : n\in \Z\}}\subset \prod_{b\in\mathscr{B}}\Z/b\Z.
\]
Consider the following list of properties:
\begin{enumerate}[(1)]
\item the window $W$ is topologically regular, i.e. $\overline{\operatorname{int}(W)}=W$,\label{B1}
\item $\mathcal{F}_{\sB}=\bigcup_{\text{finite }S\subset \sB} \cf_{\mathcal{C}_{S}}$, \label{B2}
\item $\mathcal{C}_{\infty}=\emptyset$,  \label{B3}
\item there are no $d\in\N$ and no infinite pairwise coprime set $\mathcal{C} \subseteq \N\setminus\{1\}$ such that $d \mathcal{C} \subseteq \mathscr{B}$,\label{B4}
\item $\eta=\varphi(0)$ is a Toeplitz sequence different from $\ldots0.00\ldots$,\label{B5}
\item $0 \in C_{\varphi}$ and $\varphi(0) \neq\ldots0.00\ldots$, where $C_\varphi$ is the set of continuity points of $\varphi$,\label{B6}
\item $\eta \in M$ and $\eta \neq\ldots0.00\ldots$, where $M=\overline{C_\varphi}$,\label{B7}
\item $X_{\eta}$ is minimal, i.e. $X_{\eta}=M$, and~$\operatorname{card}\left(X_{\eta}\right)>1$,\label{B8}
\item the dynamics on $(X_{\eta},\sigma)$ is a minimal almost $1$-$1$ extension of $\left(H, R_{\Delta(1)}\right)$, the rotation $\Delta(1)$ on $H$.\label{B9}
\end{enumerate}
The relations between the above conditions are as follows:
\begin{align*}
\text{$\mathscr{B}$ is taut}\Leftarrow\ &\boxed{\eqref{B1}\iff\eqref{B2}\iff\eqref{B3}\iff\eqref{B4}\iff\eqref{B5}\iff\eqref{B6}}\\
&\Downarrow\\
& \eqref{B9}\\
& \Downarrow\\
&\boxed{\eqref{B7}\iff\eqref{B8}}.
\end{align*}
Moreover, if $\overline{\Delta(\mathbb{Z}) \cap W}=W$ (in particular if $\sB$ is taut), then \eqref{B1}$-$\eqref{B9} are all equivalent.
\end{Th}
Keller in~\cite[Lemma 1]{MR4280951} showed that
\begin{equation}\label{de}
\mathcal{C}_\infty=D:=\{d\in\N :  d\cC\subset\sB \text{ for some infinite pairwise coprime set }\cC\}.
\end{equation}
Moreover, he gave the following more detailed description of the unique minimal subset of $\sB$-free systems than that in Theorem \ref{minimal_basic}.
\begin{Th}[{\cite[Corollary 5, Lemma 3]{MR4280951}}, Theorem~\ref{minimal2} in dimension one]\label{min_Keller}
For any primitive $\sB\subset\N$ let
\[
\sB^{*}:=(\sB \setminus \mathcal{M}_{\mathcal{C}_{\infty}})\cup \mathcal{C}_{\infty}^{prim}\text{ and } \eta^{*}=\raz_{\mathcal{F}_{\mathscr{B}^{*}}}.
\]
Then $\eta^{*}\leq \eta$ and it is a Toeplitz sequence. Moreover,~$X_{\eta^*}$ is a unique minimal subset of~$X_\eta$.
\end{Th}

\begin{Remark}\label{B*}
We have $\sB^*=(\sB\cup D)^{prim}$. Indeed, we need to show that
\[
(\sB\setminus\cm_D)\cup D^{prim}=(\sB\cup D)^{prim}.
\]
By~\cite[Lemma 3 a)]{MR4280951}, $\mathscr{B}^*=(\sB\setminus\cm_D)\cup D^{prim}$ is primitive. Therefore, also $\sB\setminus\cm_D$ is primitive (and disjoint with $D^{prim}$). Hence,
\[
(\sB\cup D)^{prim}=((\sB\setminus\cm_D)\cup D)^{prim}=(\sB\setminus\cm_D)^{prim}\cup D^{prim}=(\sB\setminus\cm_D)\cup D^{prim}.
\]
\end{Remark}
\begin{Remark}
Notice that by the above remark, we may define $\sB^*$ also for $\sB$ that are not primitive, just by setting $\sB^*:=(\sB\cup D)^{prim}$ (or even skip the ``prim'' in the supscipt in the above formula if we do not need that the resulting set $\sB^*$ is primitive).
\end{Remark}

Theorem~\ref{min_Keller} allows one to get rid of the technical assumption $\overline{\Delta(\mathbb{Z}) \cap W}=W$ in Theorem \ref{toeplitz}, see Theorem \ref{onedim} and Corollary \ref{toeplitzall}.
\begin{Example}
It was shown in \cite[Lemma 3.6]{MR3947636} that $\eta\in Y$ necessary for $\ov{\Delta(\Z)\cap W}=W$.
\begin{itemize}
\item
Let $\mathscr{B}=2\mathcal{P}$. Then $\mathscr{B}$ is primitive and $|\supp\ \eta \bmod 4|=2$, i.e.\ $\eta\not\in Y$. So, in this case we have $\ov{\Delta(\Z)\cap W}\neq W$.
\item
Choose $b\in\N$, $0<r<b$ and take $\mathscr{B}$ such that $b\in\mathscr{B}$ and $(r+b\N)^{prim}\subset \mathscr{B}$. Then, clearly, $|\supp\ \eta \bmod b|\leq b-2$, which gives $\eta\not\in Y$ and, again, $\ov{\Delta(\Z)\cap W}\neq W$.
\end{itemize}
\end{Example}

One of our main goals was to give an alternative proof of Theorem \ref{min_Keller} adaptable to the multidimensional case. Our proof is based on the study of a periodic structure for $\sB$-free systems. Before we comment on the tools and the related results (often interesting on their own), let us first present the full one-dimensional version of Theorem~\ref{equivalent}.
\begin{Th}[Theorem~\ref{equivalent} in dimension one]\label{onedim}
Let $\mathscr{B}\subseteq\{2,3,\ldots\}$ be primitive. The following conditions are equivalent:
\begin{enumerate}[(a)]
\item $(X_\eta,\sigma)$ is minimal,\label{min_one}
\item $\eta$ is a Toeplitz sequence different from $\ldots0.00\ldots$,\label{top_one}
\item $D=\emptyset$,
\item $X_\eta\subseteq Y$, where $Y=\{x\in\{0,1\}^\Z;\ |\operatorname{supp }x \bmod b|=b-1\text{ for any }b\in\sB\}$.\label{Y_one}
\end{enumerate}
\end{Th}
As an immediate consequence, we obtain the following corollary.
\begin{Cor}\label{toeplitzall}
Let $\mathscr{B}\subseteq\{2,3,\ldots\}$ be primitive. Conditions~\eqref{B1} to \eqref{B9} from Theorem \ref{toeplitz} are all equivalent to
\begin{enumerate}[(1)]
\setcounter{enumi}{9}
\item $X_\eta\subseteq Y$, where $Y=\{x\in\{0,1\}^\Z;\ |\operatorname{supp }x \bmod b|=b-1\text{ for any }b\in\sB\}$.
\end{enumerate}
\end{Cor}

The proof of both, Theorem~\ref{min_Keller} and Theorem~\ref{onedim}, rely on the following result that is, in fact, a part of Theorem~\ref{min_Keller} which we think that deserves to be formulated separately, with more details (see also Remark~\ref{porownanie} below).
\begin{Th}[Theorem~\ref{dwyz} in dimension one]\label{d}
Let $\mathscr{B}\subset \{2,3,\dots\}$ be primitive. Then the following are equivalent:
\begin{enumerate}[(a)]
\item $\sB=\sB^*$,
\item $D=\emptyset$,
\item $\eta$ is a Toeplitz sequence different from $\ldots0.00\ldots$.
\end{enumerate}
Moreover, we have
\[
D^*:=\{d\in\N :  d\cC\subset\sB^* \text{ for some infinite pairwise coprime set }\cC\}=\emptyset,
\]
i.e.\ $\sB^*=(\sB^*)^*$ and $\eta^*$ is Toeplitz.
\end{Th}

Recall that conditions~\eqref{top_one} and~\eqref{Y_one} from Theorem~\ref{onedim} is related to the tautness of $\sB$. More precisely, we have the following two lemmas.
\begin{Lemma}[{\cite[Proposition 5.3]{MR3731019} and \cite[Lemma 3.6]{MR3947636}}]\label{taut}
Let $\mathscr{B}\subset\N$ be primitive. If $\eta$ is a Toeplitz sequence then $\mathscr{B}$ is taut.
\end{Lemma}
\begin{Lemma}[{\cite[Corollary 4.32]{MR3803141}}]\label{inY}
If $\mathscr{B}\subset\N$ is taut then $\eta\in Y$.
\begin{Remark}
It is not so hard to see that $Y$ may not be closed, so even if $\eta\in Y$ then $X_\eta\subseteq Y$ is not automatic. Indeed, let $\mathscr{B}$ be the set of squares of all primes. Then $X_\eta$ is hereditary and it is equal to the set of admissible sequences. Then $\eta\in Y$, so for any  $b\in\sB$, there exists $m$ such that that $|\supp(\eta) \cap[-m,m] \bmod b|=b-1$. Moreover, each block that appears on $\eta$ is of positive Mirsky measure, so, in particular, it appears on $\eta$ infinitely often. This implies that $|\supp(\eta \cdot\raz_{\Z\setminus [-n,n]}) \bmod b|=b-1$ for any $n\geq 1$. In other words, $\eta \cdot\raz_{\Z\setminus [-n,n]}\in Y$. Clearly, $\eta \cdot\raz_{\Z\setminus [-n,n]}$ converges to the all-zero sequence as $n$ tends to infinity, which means that $Y$ is not closed.

In fact, it follows from Theorem~\ref{onedim} that $Y_\eta:=Y\cap X_\eta$ is closed and non-empty precisely if $\eta$ is Toeplitz. Indeed, if $Y_\eta$ is closed and non-empty then necessarily $\eta\in Y_\eta$, which yields $Y_\eta=X_\eta$ and this is equivalent to $X_\eta\subset Y$, so, by Theorem~\ref{onedim}, $\eta$ is Toeplitz. To obtain the implication in the other direction, notice that if $\eta$ is Toeplitz then $X_\eta=Y_\eta$ and it is immediate that $Y_\eta$ is closed and non-empty.
\end{Remark}
\end{Lemma}
Thus, if we are thinking of ``translating'' Theorem~\ref{onedim} to the multidimensional case, we may need a proof of the implication (b)$\implies$(c) that does not use the notion of tautness, which seems to lack a reasonable multidimensional analogue. If we look deeper in these proofs, we meet even more obstructions. Lemma \ref{inY} is a consequence of the following result.
\begin{Lemma}[{\cite[Proposition 4.31]{MR3803141}}, see also {\cite[Proposition 2.4]{DKK}}]
Assume that $\mathscr{B}\subset \N$ is taut, $d\in\N$, $1\leq r\leq d$. If
\[
r+d\Z\subseteq \cM_\sB,
\]
then there exists $b\in\sB$ such that $b\divides\gcd(r,d)$.
\end{Lemma}
Even if we ``forget'' about the problems related to the notion of tautness, the main tool in the proofs of Proposition 4.31 in \cite{MR3803141} and Proposition 2.4 in \cite{DKK} is the Dirichlet's theorem on primes in arithmetic progressions -- it has no multidimensional counterpart that could be useful for us. However, if we replace the tautness of $\sB$ by the stronger assumption that $\eta$ is a Toeplitz sequence, new tools become available and it turns out that this weaker result stated below is the right (and sufficient) way to see things in higher dimension.
\begin{Th}[Theorem~\ref{Thm_SK} in dimension one]\label{druggi2Toeplitz}
Assume that $\eta$ is a Toeplitz sequence, $d\in\N$, $1\leq r\leq d$. If
\[r+d\Z\subseteq \cM_\sB,\]
then there exists $b\in\sB$ such that $b\divides\gcd(r,d)$.
\end{Th}
In fact, the above theorem has the following consequence (cf.\ Lemma~\ref{inY}).
\begin{Cor}\label{Toeplitz_SK_one}
If $\eta$ is a Toeplitz sequence then $\eta\in Y$.
\end{Cor}
Clearly, what prevents $\eta$ from being a Toeplitz sequence is the presence of the non-periodic positions. We give the following description of the set of such positions on $\eta$.
\begin{Prop}[Proposition~\ref{propA} in dimension one]\label{lemmaA_one}
Let $(S_i)_{i\geq1}$ be a saturated filtration of a primitive set $\sB$ by finite sets. Then
\begin{equation}\label{teza1}
\Z\setminus\bigcup_{i\geq1}\Per(\eta,\lcm(S_i))=\mathcal{M}_{D}\cap \mathcal{F}_{\sB}=\Z\setminus \bigcup_{s\subset \N}\Per(\eta,s),
\end{equation}
where $D$ is as in~\eqref{de}.
\end{Prop}
The above proposition also allows us to find a periodic structure for $\eta$ in the Toeplitz case.
\begin{Th}[Theorem~\ref{thxC} in dimension one]\label{structure1}
Let $\sB\subset\N$ be primitive. Suppose $S_1\subset S_2\subset\dots\nearrow\sB$ is a saturated filtration of $\sB$ by finite sets. Assume that $\eta$ is a Toeplitz sequence. Then $(\lcm(S_i))_{i\geq1}$ is a periodic structure of $\eta$.
\end{Th}

Here is a table including the above results and their multidimensional counterparts formulated in Section~\ref{se1}.
\begin{table}[ht]
\centering
\begin{tabular}[t]{lcc}
\hline
&$\OK$&$\Z$\\
\hline
Essential minimality&Theorem~\ref{minimal2}&Theorem~\ref{min_Keller}\\
Minimality characterization&Theorem~\ref{equivalent}&Theorem~\ref{onedim}\\
Essence of being Toeplitz & Theorem~\ref{dwyz} & Theorem~\ref{d} \\
Non-periodic positions on $\eta$&Proposition~\ref{propA}&Proposition~\ref{lemmaA_one}\\
Periodic structure&Theorem~\ref{thxC}&Theorem~\ref{structure1}\\
APs in ``Toeplitz sets of multiples''&Theorem~\ref{Thm_SK}&Theorem~\ref{druggi2Toeplitz}\\
\hline
\end{tabular}
\caption{Results in higher dimension vs in dimension one.}
\end{table}%

We complete this section with a map of Section~\ref{dowody} which includes the proofs of our dimension one results.

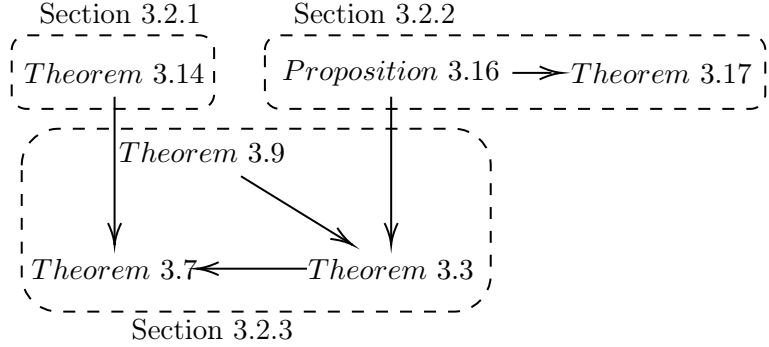
\begin{center}

\begin{tikzpicture}[x=0.75pt,y=0.75pt,yscale=-1,xscale=1]

\draw  [dash pattern={on 4.5pt off 4.5pt}] (40.67,246.5) .. controls (40.67,236.34) and (48.9,228.1) .. (59.07,228.1) -- (256.27,228.1) .. controls (266.43,228.1) and (274.67,236.34) .. (274.67,246.5) -- (274.67,301.7) .. controls (274.67,311.86) and (266.43,320.1) .. (256.27,320.1) -- (59.07,320.1) .. controls (48.9,320.1) and (40.67,311.86) .. (40.67,301.7) -- cycle ;
\draw  [dash pattern={on 4.5pt off 4.5pt}] (162,188.5) .. controls (162,184.41) and (165.31,181.1) .. (169.4,181.1) -- (404.27,181.1) .. controls (408.35,181.1) and (411.67,184.41) .. (411.67,188.5) -- (411.67,210.7) .. controls (411.67,214.79) and (408.35,218.1) .. (404.27,218.1) -- (169.4,218.1) .. controls (165.31,218.1) and (162,214.79) .. (162,210.7) -- cycle ;
\draw  [dash pattern={on 4.5pt off 4.5pt}] (35.67,188.5) .. controls (35.67,184.41) and (38.98,181.1) .. (43.07,181.1) -- (129.27,181.1) .. controls (133.35,181.1) and (136.67,184.41) .. (136.67,188.5) -- (136.67,210.7) .. controls (136.67,214.79) and (133.35,218.1) .. (129.27,218.1) -- (43.07,218.1) .. controls (38.98,218.1) and (35.67,214.79) .. (35.67,210.7) -- cycle ;


\draw (87,200) node    {$Theorem\ \ref{druggi2Toeplitz}$};
\draw (87,298.5) node    {$Theorem\ \ref{onedim}$};
\draw (225,200) node    {$Proposition\ \ref{lemmaA_one}$};
\draw (360,200) node    {$Theorem\ \ref{structure1}$};
\draw (130.75,240) node    {$Theorem\ \ref{d}$};
\draw (225,298.5) node    {$Theorem\ \ref{min_Keller}$};
\draw (175,163.1) node [anchor=north west][inner sep=0.75pt]   [align=left] {Section 3.2.2};
\draw (94,322.1) node [anchor=north west][inner sep=0.75pt]   [align=left] {Section 3.2.3};
\draw (48,163.1) node [anchor=north west][inner sep=0.75pt]   [align=left] {Section 3.2.1};
\draw    (285.5,200) -- (305,200) ;
\draw [shift={(310,200)}, rotate = 180] [color={rgb, 255:red, 0; green, 0; blue, 0 }  ][line width=0.75]    (10.93,-3.29) .. controls (6.95,-1.4) and (3.31,-0.3) .. (0,0) .. controls (3.31,0.3) and (6.95,1.4) .. (10.93,3.29)   ;
\draw    (87,212) -- (87,284.5) ;
\draw [shift={(87,286.5)}, rotate = 270] [color={rgb, 255:red, 0; green, 0; blue, 0 }  ][line width=0.75]    (10.93,-3.29) .. controls (6.95,-1.4) and (3.31,-0.3) .. (0,0) .. controls (3.31,0.3) and (6.95,1.4) .. (10.93,3.29)   ;
\draw    (130,298.5) -- (183,298.5) ;
\draw [shift={(128,298.5)}, rotate = 0] [color={rgb, 255:red, 0; green, 0; blue, 0 }  ][line width=0.75]    (10.93,-3.29) .. controls (6.95,-1.4) and (3.31,-0.3) .. (0,0) .. controls (3.31,0.3) and (6.95,1.4) .. (10.93,3.29)   ;
\draw    (225,212) -- (225,284.5) ;
\draw [shift={(225,286.5)}, rotate = 270] [color={rgb, 255:red, 0; green, 0; blue, 0 }  ][line width=0.75]    (10.93,-3.29) .. controls (6.95,-1.4) and (3.31,-0.3) .. (0,0) .. controls (3.31,0.3) and (6.95,1.4) .. (10.93,3.29)   ;
\draw    (150.08,252) -- (203.97,285.45) ;
\draw [shift={(205.67,286.5)}, rotate = 211.83] [color={rgb, 255:red, 0; green, 0; blue, 0 }  ][line width=0.75]    (10.93,-3.29) .. controls (6.95,-1.4) and (3.31,-0.3) .. (0,0) .. controls (3.31,0.3) and (6.95,1.4) .. (10.93,3.29)   ;

\end{tikzpicture}
%

\end{center}

\subsection{Proofs}\label{dowody}
\subsubsection{Arithmetic progressions in $\mathcal{M}_\sB$: proof of Theorem~\ref{druggi2Toeplitz}}
The proof of Theorem~\ref{druggi2Toeplitz} relies on two lemmas, whose proofs will be given in a moment.
\begin{Lemma}\label{lem1}
Let $d,r\in\N$. Then there exists an infinite pairwise coprime set $\mathcal{C}$ such that
\[
\gcd(r,d)\mathcal{C}\subseteq r+d\Z.
\]
\end{Lemma}
\begin{Lemma}\label{lem2}
Suppose that $\sB$ is such that $D=\emptyset$. If $d\mathcal{C}\subseteq \cM_\sB$ for some $d\in\N$ and some infinite pairwise coprime set $\mathcal{C}$ then $d\in\cM_\sB$.
\end{Lemma}
\begin{proof}[Proof of Theorem~\ref{druggi2Toeplitz}]
Suppose that $r+d\Z \subseteq \cM_\sB$. Then, by Lemma~\ref{lem1}, we have
\[
\gcd(r,d)\mathcal{C} \subseteq r+d\Z\subseteq \cM_\sB
\]
for some infinite pairwise coprime set $\mathcal{C}$. Since $\eta$ is a Toeplitz sequence, it follows from Theorem~\ref{toeplitz} (\eqref{B5}$\Leftrightarrow$\eqref{B4}) that $D=\emptyset$. Therefore, by Lemma~\ref{lem2}, we have $\gcd(r,d) \in \cM_\sB$, which completes the proof.
\end{proof}
\paragraph{Proofs of lemmas}
\begin{proof}[Proof of Lemma~\ref{lem1}]
It suffices to show that if $\gcd(r,d)=1$ then $r+d\Z$ contains an infinite pairwise coprime set. Notice that for any finite subset of primes $P$, there exists $k\in \Z$ such that
\begin{equation}\label{*p}
\gcd(r+kd, \prod_{p\in P}p)=1
\end{equation}
(take, for example, $k=\prod_{p\in P, p\ndivides r}p$ or $k=1$ if $\{p\in P, p\ndivides r\}$ is empty). Suppose now that $r+d\Z$ does not contain an infinite pairwise coprime subset. Then there is a pairwise coprime set $K$ of the maximal cardinality contained in~$r+d\Z$. Clearly, any element of $r+d\Z$ shares a factor with some element of this set. It remains to notice that~\eqref{*p} does not hold for $P$ being the set of all prime factors of elements of $K$  to obtain a contradition.
\end{proof}
\begin{proof}[Proof of Lemma~\ref{lem2}]
Let $\mathcal{C}=\{c_i\}_{i\geq1}$. Since $d\mathcal{C}\subseteq\cm_\sB$, for any $i\geq1$ there exist $b_i\in\sB$ and $k_i\in\N$ such that
\[dc_i=b_ik_i.
\]

Suppose first that there exists $b\in\sB$ such that
\[
b_i=b \text{ for infinitely many }i\geq 1.
\]
Then since $c_i$, $i\geq1$, are pairwise coprime, $b\divides d$.

Suppose now that each $b\in\sB$ appears among $b_i$'s at most a finite number of times. In particular, passing to a subsequence if necessary, we may assume that $b_i$'s are pairwise different (the corresponding $c_i$'s still form an infinite pairwise coprime set). Since $d$ has only finitely many factors, we can assume that $\gcd(d,k_i)$ is constant. Then
\[
b_i=\frac{d}{\gcd(d,k_i)}\cdot \frac{c_i \cdot \gcd(d,k_i)}{k_i}.
\]
Since $c_i$'s are pairwise coprime and $\frac{c_i \cdot \gcd(d,k_i)}{k_i}\divides c_i$, it follows immediately that
also $\frac{c_i \cdot \gcd(d,k_i)}{k_i}$ ($i\geq 1$) are pairwise coprime. This completes the proof.
\end{proof}
\subsubsection{Non-periodic positions (proof of Proposition~\ref{lemmaA_one}) and periodic structure (proof of Theorem~\ref{structure1})}
The structure of this section is fairly simple, so we just proceed with the proofs.
\begin{proof}[Proof of Proposition~\ref{lemmaA_one}]
It suffices to show that
\begin{equation}\label{dwie_inkluzje}
\Z\setminus \bigcup_{i\geq 1}\Per(\eta,\lcm(S_i)) \subseteq \mathcal{M}_D\cap\cf_\sB \subseteq \Z\setminus \bigcup_{s\in\N}\Per(\eta,s).
\end{equation}
Clearly, $\Z\setminus \bigcup_{s\in\N}\Per(\eta,s)\subseteq \Z\setminus \bigcup_{i\geq 1}\Per(\eta,\lcm(S_i))$ and therefore the inclusions must be, in fact, equalities.

\paragraph{First inclusion.} Take $n\not\in\bigcup_{i\geq 1}\Per(\eta,\lcm(S_i))$. Then $n\in\cf_\sB$. Moreover, for any $k\geq 1$ there exists $t\in \Z$ such that $n+t\lcm(S_k) \in\mathcal{M}_\mathscr{B}$, so there exists $b_k$ such that
\[
b_k\mid n+t\lcm(S_k).
\]
Notice that $b_k\notin S_k$ (if $b_k\in S_k$ then $b_k\mid n$ which contradicts $n\in\cf_\sB$). Since $(S_k)_{k\geq1}$ is saturated, $\gcd(b_k,\lcm(S_k))<b_k$. Notice that $\gcd(b_k,\lcm(S_k))\mid n$. Since $n$ has only finitely many divisors, there exists $d\in\N$ and a subsequence $(k_\ell)\subseteq\N$ such that $\gcd(b_{k_\ell},\lcm(S_{k_\ell}))=d$ for some $d\in\N$ (and $d\divides n$).

Let $\ell_2>\ell_1:=1$ be such that
\begin{equation}\label{maj1}
b_{k_{\ell_1}}\in S_{\ell_{k_2}} \text{ (in particular, $b_{k_{\ell_1}}\divides \lcm(S_{\ell_{k_2}})$)}.
\end{equation}
By the choice of $(k_\ell)$, we have $d\divides b_{k_{\ell_1}},b_{k_{\ell_2}}$, so
\[
d\divides \gcd(b_{k_{\ell_1}},b_{k_{\ell_2}}).
\]
On the other hand, using~\eqref{maj1}, we obtain
\[
\gcd(b_{k_{\ell_1}},b_{k_{\ell_2}})\divides \gcd(\lcm(S_{k_{\ell_2}}),b_{k_{\ell_2}})=d.
\]
Thus, $\gcd(b_{k_{\ell_1}},b_{k_{\ell_2}})=d$.

Suppose that we have constructed $1=\ell_1<\ell_2<\dots<\ell_m$ such that
\begin{equation}\label{maj2}
\gcd(b_{k_{\ell_i}},b_{k_{\ell_{i'}}})=d \text{ for any }1\leq i<i'\leq m.
\end{equation}
There exists $\ell_{m+1}>\ell_m$ such that
\begin{equation}\label{maj3}
b_{k_{\ell_1}},b_{k_{\ell_2}},\dots,b_{k_{\ell_m}}\in S_{k_{\ell_{m+1}}}.
\end{equation}
It follows by the choice of $(k_\ell)$ that $d\divides b_{k_{\ell_i}}$ for $1\leq i\leq m+1$. This gives
\[
d\divides \gcd(b_{k_{\ell_i}},b_{k_{\ell_{i'}}}) \text{ for any }1\leq i<i'\leq m+1.
\]
On the other hand, using~\eqref{maj3}, we obtain $\gcd(b_{k_{\ell_i}},b_{k_{\ell_{m+1}}})\divides\gcd(\lcm(S_{k_{\ell_{m+1}}}),b_{k_{\ell_{m+1}}})=d$. Thus,
\[
\gcd(b_{k_{\ell_i}},b_{k_{\ell_{i'}}})=d \text{ for any }1\leq i<i'\leq m+1
\]
and the above inductive procedure therefore yields a sequence $1=\ell_1<\ell_2<\dots$ such that
\[
\gcd(b_{k_{\ell_i}},b_{k_{\ell_{i'}}})=d \text{ for any }i\neq i;
\]
Let $c_i:=\frac{b_{k_i}}{d}$. It remains to notice that then $\mathcal{C}=\{c_i : i\geq 1\}$ is pairwise coprime and we have $d\mathcal{C}\subseteq \mathscr{B}$. Since $d\divides n$, it follows that $n\in\mathcal{M}_D$.

\paragraph{Second inclusion.}
Take $n\in\mathcal{M}_D\cap\mathcal{F}_\mathscr{B}$. There exists $d$ and an infinite pairwise coprime set $\mathcal{C}$ such that
\[
d\divides n \text{ and }d\mathcal{C}\subseteq\mathscr{B}.
\]
Suppose that $n\in\text{Per}(\eta,s)$ for some $s\in\N$. In other words, we have $n+s\Z\subseteq \mathcal{F}_\mathscr{B}$. Since $d\divides n$, it follows that
\begin{equation}\label{inclusion3_one}
d\Z\cap (n+s\Z)=n+\lcm(s,d)\Z\subseteq n+s\Z\subseteq\mathcal{F}_\sB.
\end{equation}
Take $c\in\mathcal{C}$ coprime to $s$ and $d$. Then
\[
\emptyset\neq (n+\lcm(s,d)\Z)\cap dc\Z \subseteq (n+\lcm(s,d)\Z)\cap\mathcal{M}_\sB
\]
as $dc\in\sB$. This, however, contradicts~\eqref{inclusion3_one}.
\end{proof}

\begin{proof}[Proof of Theorem \ref{structure1}]
For $i\geq1$ put $s_i:=\lcm(S_i)$. It follows from the equivalence \eqref{B5} $\Leftrightarrow$ \eqref{B4} in Theorem~\ref{toeplitz} that $D=\emptyset$ and thus $\mathcal{M}_D\cap \mathcal{F}_\sB=\emptyset$. Together with Proposition~\ref{lemmaA_one} this yields
\[
\bigcup_{i\geq1}\Per(\eta,s_i)=\Z.
\]
Clearly, $s_i\divides s_{i+1}$, so it remains to show that $s_i$ is an essential period of $\eta$.
Let $s<s_i$. Then there exists $b\in S_i$ such that
\begin{equation}\label{bj_s}
b \ndivides s.
\end{equation}
We have $b \in \operatorname{Per}(\eta,s_i)$. Suppose that $b\in \operatorname{Per}(\eta, s)$, i.e.\
\[
b+s \mathbb{Z} \subseteq \mathcal{M}_{\mathscr{B}}
\]
By Theorem~\ref{druggi2Toeplitz} (for $d=s$, $r=b$), we get that for some $b' \in \mathscr{B}$,
\[
b' \divides \gcd(b,s)\divides b.
\]
Since $\mathscr{B}$ is primitive, $b'=b$, so $b\divides s$. But this contradicts \eqref{bj_s}, which completes the proof.
\end{proof}
\begin{Example}
Let $\sB=\{2^ic_i;\ i\geq 1\}$, where $c_i\geq 2$ for $i\geq 1$. In Example 3.1 in \cite{MR3803141} it is shown that $\Z=\bigcup_{i\geq1}\Per(\eta,2^{i+1}c_1\ldots c_i)$. Let $S_k=\{2^ic_i: 1\leq i\leq k\}$. By Theorem~\ref{structure1}, we have $\Z=\bigcup_{i\geq1}\Per(\eta,2^ic_1,\ldots,c_i)$.
\end{Example}
\begin{Example}
The first example of $0$-$1$ nonperiodic Toeplitz sequence comes from Garcia and Hedlund \cite{MR0074810}. The sequence is given by the following rule: in each odd-numbered step we fill every second available
position with $0$, in each even-numbered step we fill every second available position
with $1$. The corresponding periodic structure
is $(2^m)_{m\geq1}$. On the other hand, it follows by Theorem~\ref{structure1} that $(2^m)_{m\geq1}$ cannot be a periodic structure of any Toeplitz sequence given by some $\eta^*$ (corresponding to $\sB^*$). Since any two Toeplitz sequences from the same Toeplitz subshift have the same periodic structure, it follows that the Garcia-Hedlund sequence cannot be an element of a $\sB$-free subshift. Indeed, by Theorem~\ref{min_Keller}, it would be then an element of the corresponding $\sB^*$-free subshift which yields a contradiction.
%
\end{Example}

\subsubsection{Minimality: proof of Theorems~\ref{min_Keller} and~\ref{onedim}}
We will prove first Theorem~\ref{d}, necessary for the proof of Theorem~\ref{min_Keller}. Then we make some comments on set $\mathscr{B}^*$ and compare it to the taut set $\mathscr{B}'$ from~\cite{MR3803141} giving the same Mirsky measure as $\mathscr{B}$. Then we pass to the proof of Theorem~\ref{min_Keller}. Finally, we prove Theorem~\ref{onedim}, using both, Theorem~\ref{d} and Theorem~\ref{min_Keller}.
\begin{proof}[Proof of Theorem~\ref{d}]
We will first show (a) $\iff$ (b). Clearly, if $D=\emptyset$ then $\sB=\sB^*$. For the other direction, notice that
\[
\sB=\sB^*=(\sB\setminus \mathcal{M}_D)\cup D^{prim} \text{ implies }D^{prim}\subset \sB.
\]
However, by the definition of $D$, for any $d\in D$, some non-trivial multiple of $d$ is a member of $\sB$, as $d\mathcal{C}\subset \sB$ where $\mathcal{C}$ is inifinite and pairwise coprime. If $D\neq\emptyset$ this contradicts the primitivity of $\sB$.

Let us show that $D^*=\emptyset$ (i.e.\ $\eta^*$ is Toeplitz, so (a) $\implies$ (c)). Suppose that there is some $d\in\N$ and some infinite pairwise coprime set $\cC$ such that
\[
d\mathcal{C}\subseteq \sB^*=(\sB\setminus \mathcal{M}_D)\cup D^{prim}.
\]
Then, without loss of generality (taking a smaller but still infinite and pairwise coprime set $\cC$), one of the following holds:
\begin{enumerate}[(A)]
\item
$d\mathcal{C}\subseteq \sB\setminus \mathcal{M}_D$,
\item
$d\mathcal{C}\subseteq D^{prim}$.
\end{enumerate}
If (A) holds then $d\mathcal{C}\subset \sB$, so $d\in D$ and we obtain
\[
d\mathcal{C}\subset \mathscr{B}\setminus \mathcal{M}_D\subset \mathscr{B}\setminus d\Z,
\]
which yields a contradiction. Suppose now that (B) holds. Let $\mathcal{C}=\{c_1,c_2,\dots\}$ and let $\mathcal{A}_i$ for $i\geq 1$ be infinite, pairwise coprime and such that $dc_i\mathcal{A}_i\subset \sB$. We can choose $a_i\in\mathcal{A}_i$, $i\geq 1$ so that $\{c_ia_i:i\geq 1\}$ is infinite and pairwise coprime (we just use that each $\mathcal{A}_i$ is infinite and pairwise coprime). This gives again $d\in D$. This contradicts (B) (by the primitivity of $D^{prim}$). We conclude that indeed $D^*=\emptyset$.

Finally, we will show (c) $\implies$ (a). If $\eta$ is Toeplitz then $\mathcal{F}_\sB\cap \mathcal{M}_D=\emptyset$ since (by Proposition~\ref{lemmaA_one}) it is the set of non-periodic positions on $\eta$. Therefore, $\mathcal{M}_D\subseteq \mathcal{M}_\sB$ and it follows immediately (see Remark~\ref{B*}) that $\mathcal{M}_{\sB^*}=\mathcal{M}_{\sB}$. This yields $\sB=\sB^*$ by the primitivity of $\sB$ and $\sB^*$.

\end{proof}

\begin{Remark}\label{porownanie}
In~\cite{MR3803141}, a certain procedure was described to modify $\sB$ to $\sB'$, so that:
\begin{itemize}
\item $\sB'$ is taut,
\item $\eta'\leq\eta$ (where $\eta'=\mathbf{1}_{\mathcal{F}_{\sB'}}$),
\item $\nu_{\eta}=\nu_{\eta'}$.
\end{itemize}
The idea was very similar to the one used to produce $\sB^*$. We recall it here.
\begin{itemize}
\item
Suppose that $\sB$ is not taut, let $c_1$ be the smallest natural number such that there exists a Behrend set $\mathcal{A}_1$ such that $c_1\mathcal{A}_1\subset \sB$. Replace $\sB$ with $\sB\setminus c_1\Z \cup \{c_1\}$.
\item
If now $\sB$ is taut, we stop. If not, we take the smallest natural number $c_2$ such that there exists a Behrend set $\mathcal{A}_2$ such that $c_2\mathcal{A}_2\subset \sB$ and replace the original $\sB$ with $\sB\setminus (c_1\Z\cup c_2\Z)\cup\{c_1,c_2\}$.
\end{itemize}
Either the above procedure ends after a finite number of steps or we arrive at the modified $\sB$ of the form
\[
\sB'=\sB \setminus \bigcup_{i\geq 1}c_i\Z \cup \{c_i : i\geq 1\},
\]
where $\mathcal{A}_i$, $i\geq 1$ are Behrend sets such that $c_i\cA_i\subseteq\sB \setminus \bigcup_{j=1}^{i-1}c_j\Z \cup \{c_1,c_2,\ldots,c_{i-1}\}$ . It was shown in Lemma~4.10 in~\cite{MR3803141} that such $\sB'$ fulfills the requirements listed in the beginning of this remark. In particular, this means that one could equivalently define $\sB'$ as follows:
\begin{equation}\label{ce}
\sB'':=(\sB\setminus \mathcal{M}_C) \cup C^{prim}=(\sB\cup C)^{prim},
\end{equation}
where
\[
C=\{c\in \N : \text{there exists a Behrend set $\mathcal{A}$ such that }c\mathcal{A}\subset \sB\}.
\]
Indeed, we need to show that $C^{prim}=\{c_i : i\geq 1\}$ (the proof of the second equality above goes along the same lines as in Remark~\ref{B*}). The inclusion $C^{prim}\supset\{c_i : i\geq 1\}$ follows directly by the construction of $\{c_i : i\geq 1\}$. Suppose that we do not have an equality, i.e.\ there exists $c\in C^{prim}$ such that $c\neq c_i$ for all $i\geq1$. Let $n\in\N$ be such that $c_{n-1}<c<c_n$. Then $c\cA\not\subset\sB\setminus\bigcup_{i=1}^{n-1}c_i\Z$ for any Behrend set $\cA$ but $\cA_c:=\{\frac{b}{c}: b\in\sB, c\mid b\}$ is Behrend. Notice that
\[
\cA_c=\left(\cA_c\cap\left\{\frac{b}{c}: b\in\sB\setminus\bigcup_{i=1}^{n-1}c_i\Z, c\mid b\right\}\right)\cup \bigcup_{i=1}^{n-1} \left(\cA_c\cap\left\{\frac{b}{c}:b\in \sB\cap c_i\Z,c\mid b\right\}\right).
\]
By the above, the first set in the above sum is not Behrend. Recall that by \cite[Corollary 0.14]{MR1414678} a finite union of sets is Behrend, provided that at least one of them is Behrend. It follows immediately that there exists $1\leq i\leq n-1$ such that $\cA_c\cap\left\{\frac{b}{c}:b\in c_i\Z,c\mid b\right\}$ is Behrend. However,
\[
\cA_c\cap\left\{\frac{b}{c}:b\in \sB\cap c_i\Z,c\mid b\right\}\subset\frac{c_i}{\gcd(c,c_i)}\Z.
\]
Therefore,
\[
1=d\left(\cM_{\cA_c\cap\left\{\frac{b}{c}:b\in c_i\Z,c\mid b\right\}}\right)\leq d\left(\frac{c_i}{\gcd(c,c_i)}\Z\right)=\frac{\gcd(c,c_i)}{c_i}.\] Hence $\gcd(c,c_i) =c_i$ which contradicts $c\in C^{prim}\supset\{c_i:i\geq1\}$.

Now, we claim that we have the following inclusions:
\[
X_{\eta^*}\subset X_{\eta'}\subset X_\eta.
\]
Indeed, for the second inclusion, see (27) in~\cite{MR4289651}. The first inclusion follows by the fact that $X_{\eta^*}$ is the unique minimal subset of $X_\eta$. So we must have $X_{\eta^*}=X_{\eta^*}\cap X_{\eta'}$ and thus $X_{\eta^*}\subset X_{\eta'}$. This also implies that $(\sB')^*=\sB^*$. Indeed, $X_{\eta^*}$ is not only the unique minimal subset of $X_\eta$ but also of $X_{\eta'}$. Therefore, we have $X_{\eta^*}=X_{(\eta')^*}$. Both, $\sB^*$ and $(\sB')^*$ are taut and $(\sB')^*=\sB^*$ follows immediately by Corollary 4.36 in~\cite{MR3803141}.

To sum up the above discussion, let us make the following observation. Up to applying ``prim'' at the end,
\begin{itemize}
\item
the procedure that outputs $\sB^*$ means replacing the rescaled copies of infinite pairwise coprime subsets of $\sB$ with the set of scales (cf.\ Remark~\ref{B*}),
\item
the procedure that outputs $\sB'$ means replacing the rescaled copies of Behrend subsets of $\sB$ with the set of scales (cf.\ \eqref{ce}).
\end{itemize}
Recall now that each Behrend set contains an infinite pairwise coprime subset as the corresponding $\sB$-free subshift is proximal (see \cite[Theorem 3.7]{MR3803141}). This tells us immediately that
$C\subset D$ so $\cM_\sB\subseteq\cM_{\sB'}\subseteq\cM_{\sB^*}$.
\end{Remark}


\begin{proof}[Proof of Theorem~\ref{min_Keller}]
Since, by Remark~\ref{B*}, $\sB^*=(\sB\cup D)^{prim}$, we have $\eta^*\leq\eta$. Moreover, $\eta^*$ is Toeplitz, by Theorem~\ref{d}.

Now, we will prove that $\eta^*\in X_\eta$, i.e.\ that for any $N\geq 1$ the block $\eta^*|_{[-N,N]}$ appears on $\eta$. Fix $N\geq 1$ and let $i\geq1$ be sufficiently large to ensure that
\begin{equation}\label{righti_one}
 \cm_D\cap\cf_\sB\cap [-N,N]=(\Z\setminus\Per(\eta,\lcm(S_i)))\cap [-N,N].
\end{equation}
Recall that by Proposition~\ref{lemmaA_one}, we have $\cm_D\cap\cf_\sB=\Z\setminus\bigcup_{i\geq 1}\Per(\eta,\lcm(S_i))$,
and $\Per(\eta,\lcm(S_i))$ grows, as $i$ grows. Moreover, by Proposition~\ref{lemmaA_one}, $\cm_D\cap\cf_\sB\cap [-N,N]$ are all non-periodic positions on $\eta$ restricted to $[-N,N]$. Let $t_N$ be the cardinality of $\cm_D\cap\cf_\sB\cap [-N,N]$ and denote the elements of this set by $I_1,\dots, I_{t_N}$. It follows by the definition of $D$ that for any $1\leq j\leq t_N$, there exists $d_j\in \N$ and infinite pairwise coprime set $\mathcal{C}_j\subset \N$ such that
\[
d_j\divides I_j \text{ and }d_j\mathcal{C}_j\subset \sB.
\]
Let
%
\[
L:=\lcm(\lcm(S_i),d_1,\ldots,d_{t_N}).
\]
Let $c_1\in\mathcal{C}_1$ be coprime to $L$ (such a number exists since $L$ has finitely many factors and~$\mathcal{C}_1$ is infinite pairwise coprime). Then $\gcd(L,d_1c_1)=d_1\divides I_1$ and therefore we can find $k_1\in \Z$ such that
\[
Lk_1+I_1\equiv 0 \bmod d_1c_1.
\]
Let $c_2\in\mathcal{C}_2$ be coprime to $L$ and $c_1$ (again, such a number exists because $L$ and~$c_1$ have finitely many factors and $\mathcal{C}_2$ is infinite pairwise coprime). Since $d_2\mid L$ and $L$, $c_1$ and~$c_2$ are pairwise coprime, we have $\gcd(Lc_1,d_2c_2)=d_2\divides Lk_1+I_2$. Therefore, we can find $k_2\in\Z$ such that
\[
Lc_1k_2+Lk_1+I_2\equiv 0 \bmod d_2c_2.
\]
Notice that since $d_1\mid L$, we have
\[
Lc_1k_2+Lk_1+I_1\equiv 0\bmod d_1c_1.
\]
Let $c_3\in\mathcal{C}_3$ be coprime to $L$, $c_1$ and~$c_2$. Since $d_3\divides L$ and $L$, $c_1,c_2$ and $c_3$ are pairwise coprime, we have $\lcm(Lc_2c_1,d_3c_3)=d_3 \divides Lc_1k_2+Lk_1+I_3$. Therefore, we can find $k_3\in\Z$ such that
\begin{align*}
Lc_2c_1k_3+Lc_1k_2+Lk_1+I_3\equiv 0 \bmod d_3c_3.
\end{align*}
Notice that since $d_\ell c_\ell\mid Lc_2c_1$ for $\ell=1,2$, we have
\begin{align*}
Lc_1c_2k_3+Lc_1k_2+Lk_1+I_\ell\equiv0 \bmod d_\ell c_\ell\text{ for }\ell=1,2.
\end{align*}
By repeating the above arguments, we obtain $k_j\in\Z$ for $j=1,2,\ldots,t_N$ such that
\begin{align}\label{kok_one}
Lc_{t_{N-1}}\ldots c_1k_{t_N}+Lc_{t_{N-2}}\ldots c_1k_{t_{N-1}}+\ldots+Lc_1k_2+Lk_1+I_j\equiv0 \bmod d_j c_j.
\end{align}
for $j=1,2,\ldots,t_N$. Put
\[
M_N:=Lc_{t_{N-1}}\ldots c_1k_{t_N}+Lc_{t_{N-2}}\ldots c_1k_{t_{N-1}}+\ldots+Lc_1k_2+Lk_1.
\]

We will show that
\begin{equation}\label{okresowe_one}
\eta_{n+M_N}=\eta_n=\eta^*_n \text{ for }n\in\Per(\eta,\lcm(S_i))\cap [-N,N].
\end{equation}
Take $n\in\Per(\eta,\lcm(S_i))\cap [-N,N]$. Since  $\lcm(S_i)\divides L \divides M_N$, it follows that
\begin{equation}\label{AA}
\eta_{n+M_N}=\eta_n \text{ for }n\in\Per(\eta,\lcm(S_i))\cap [-N,N]
\end{equation}
and $n+M_N\in \Per(\eta,\lcm(S_i))$. It follows by Proposition~\ref{lemmaA_one} that $n,n+M_N\not\in \cm_D\cap\cf_\sB$ (the latter set is the set of all non-periodic positions on $\eta$). In other words,
\begin{equation*}
n,n+M_n\in \mathcal{F}_D \cup \cm_\sB=(\mathcal{F}_D\cap \mathcal{F}_{\sB})\cup \mathcal{M}_{\sB}.
\end{equation*}
Recall that by Remark~\ref{B*}, we have $\sB^*=(\sB\cup D)^{prim}$, so $\mathcal{M}_{\sB^*}=\mathcal{M}_{\sB}\cup\mathcal{M}_{D}$. Equivalently, $\mathcal{F}_{\sB^*}=\mathcal{F}_{\sB}\cap\mathcal{F}_D$.
Therefore,
\begin{itemize}
\item
if $n\in \mathcal{F}_D\cap \mathcal{F}_\sB$ then $n\in\mathcal{F}_{\sB^*}$,
\item
if $n\in\mathcal{M}_{\sB}$ then $n\in\mathcal{M}_{\sB^*}$.
\end{itemize}
Hence,~\eqref{okresowe_one} indeed holds.

Now, we will show that
\begin{equation}\label{nieokresowe_one}
\eta_{n+M_N}=\eta^*_n \text{ for }n\in(\Z\setminus\Per(\eta,\lcm(S_i)))\cap [-N,N].
\end{equation}
Since $d_jc_j\in\sB$ and $d_jc_j\divides I_j+M_N$, it follows by \eqref{kok_one} that
\begin{equation}\label{BB}
\eta_{I_j+M_N}=0 \text{ for all }1\leq j\leq t_N.
\end{equation}
Moreover, it follows by $I_j\in\cm_D\subseteq \cm_{\sB^*}$ (recall again~Remark \ref{B*}), that we also have
\begin{equation*}
\eta^*_{I_j}=0\text{ for all }1\leq j\leq t_N.
\end{equation*}
Therefore,~\eqref{nieokresowe_one} indeed holds. Combining~\eqref{okresowe_one} and~\eqref{nieokresowe_one}, we conclude that
\begin{equation}\label{etai*_one}
\eta|_{[-N,N]+M_N}=\eta^*|_{[-N,N]}.
\end{equation}

Notice that in the above arguments we used only that
\[
L\divides M_N \text{ and }d_jc_j\divides M_N+I_j
\]
(to obtain~\eqref{AA} and~\eqref{BB}, respectively). Thus, by~\eqref{kok_one},
\[
\eta|_{[-N,N]+M_N}=\eta|_{[-N,N]+M_N+s} \text{ for any }s\in Lc_1\ldots c_{t_N}\Z.
\]
In particular, $\{s\in\Z;\ \eta|_{[-N,N]+M_N}=\eta|_{[-N,N]+M_N+s}\}$ is syndetic. Moreover, the blocks $(\eta|_{[-N,N]+M_N})_{N\geq1}$ have different lengths, so they are pairwise different. By \cite[Corollary 2.17]{MR3803141} (see Corollary~\ref{dlapod}), $X_\eta$ has a unique minimal subset.
\end{proof}

\begin{proof}[Proof of Theorem \ref{onedim}]
Obviously, \eqref{top_one} $\Rightarrow$ \eqref{min_one}.

\eqref{min_one} $\Rightarrow$ \eqref{top_one}.
It follows from~\eqref{min_one} and from Theorem~\ref{min_Keller} that $X_{\eta^*}=X_{\eta}$. Suppose that $\eta^*<\eta$, i.e.\ there is $n\in\Z$ such that $0=\eta^*_n<\eta_n=1$. Let $(p_k)_{k\in\N}$ be a periodic structure of $\eta^*$ and let $k\in\N$ be such that $\eta^*|_{n+p_k\Z}=0$. It follows from~\eqref{downar} that
\begin{equation}\label{ble}
|\text{Per}(\eta^*,p_k,0)\cap [0,p_k-1]|=|\text{Per}(\eta,p_k,0)\cap[0,p_k-1]|
\end{equation}
as $\eta\in X_{\eta^*}$.  On the other hand, since $\eta^*\leq \eta$, it follows immediately that
\[
\text{Per}(\eta^*,p_k,0) \supset \text{Per}(\eta,p_k,0).
\]
It remains to notice that
$n\in \text{Per}(\eta^*,p_k,0)\setminus \text{Per}(\eta,p_k,0)$, so
\[
|\text{Per}(\eta^*,p_k,0) \cap [0,p_k-1]| > |\text{Per}(\eta,p_k,0)\cap [0,p_k-1]|,
\]
which contradicts~\eqref{ble}. This yields $\eta=\eta^*$, i.e.\ $\eta$ itself is a Toeplitz sequence (\eqref{top_one} holds).

\eqref{top_one} $\Rightarrow$ \eqref{Y_one} Suppose $\eta$ is a Toeplitz sequence. 
For any $b\in\sB$ and~$0\leq s_b\leq b$ let \[Y^b_{\geq s_b}=\{\bm{x}\in \{0,1\}^\Z;\ |\text{supp }\bm{x}\bmod b|\leq b-s_b\}\] and \[Y^b_{s_b}=\{\bm{x}\in \{0,1\}^\Z;\ |\text{supp }\bm{x}\bmod b|=b-s_b\}.\] Then $Y^b_{\geq s_b}$ is closed and~$\sigma$-invariant. By the minimality of $(X_\eta,\sigma)$, it follows that for any $0\leq s_b\leq b$, $b\in\sB$, we have
\[
X_\eta \cap Y^b_{\geq s_b}=X_\eta\text{ or }X_\eta \cap Y^b_{\geq s_b}=\emptyset.
\]
By Corollary \ref{Toeplitz_SK_one}, $\eta\in Y$, so for $s_b\geq 2$ we have
\[
X_\eta \cap Y^b_{\geq s_b}=\emptyset.
\]
Since for any $b\in\sB$ we have
\[
X_\eta=X_\eta \cap \left(\bigcup_{1\leq s_b\leq b}Y^b_{s_b}\right),
\]
it follows immediately that $X_\eta=X_\eta\cap Y$.

\eqref{Y_one} $\Rightarrow$ \eqref{top_one} Suppose $X_\eta\subseteq Y$. By Theorem~\ref{min_Keller}, we have $X_{\eta^*}\subseteq X_\eta$, so, in particular, $\eta^*\in X_\eta$ is an element of $Y$. Since $\eta^*$ is a Toeplitz sequence, it follows by Corollary~\ref{Toeplitz_SK_one} that \[\eta^*\in Y^*,\]where
\[
Y^*=\{\bm{x}\in\{0,1\}^{\Z};\ |\supp \ \bm{x}\bmod b^*|=b^*-1 \text{ for any }b^*\in\sB^*\}.
\]
Suppose there exists $b\in\sB\setminus \sB^*$. Let $b^*\in\sB$ be such that $b^*\mid b$. Then
\[
|\supp \ \eta^* \bmod b^*|=b^*-1
\]
and we conclude that
\[
\left|\supp \ \eta^* \bmod b\right|=\frac{b}{b^*}\left|\supp \ \eta^* \bmod b^*\right|=\frac{b}{b^*}\left(b^*-1\right)<b-1.
\]
Thus, $\eta^*\not\in Y$, which we know that is not true. It follows that $\sB=\sB^*$ and therefore~$\eta=\eta^*$ is a Toeplitz sequence.
\end{proof}
\begin{proof}[Proof of Corollary \ref{Toeplitz_SK_one}]
Suppose that $\eta$ is a Toeplitz sequence and $r+b\Z\subseteq\cm_\sB$ for some $b\in\sB$ and $1\leq r\leq b$. Then by Theorem \ref{druggi2Toeplitz} (for $d=b$), there is $b'\in\sB$ such that $b'\mid\gcd(r,b)$. Since $\sB$ is primitive, $b=b'$. Hence $r=b$. So $\eta\in Y$.
\end{proof}
\begin{Remark}
Arguments used in the proof \eqref{top_one}$\Rightarrow$\eqref{Y_one} of Theorem \ref{onedim} comes from Example 3.16 in \cite{MR3803141}.
\end{Remark}

\section{Minimality and periodic structure of $\mathfrak{B}$-free systems}\label{se4}
\subsection{Arithmetic progressions}

The main goal of this section is to prove Theorem~\ref{Thm_SK}. We will prove first its easier version, valid for principal ideals.
\begin{Th}\label{cor_toeplitz}
Let $\mathfrak{B}$ be a primitive collection of ideals in $\OK$ such that $\mathfrak{D}=\emptyset$, $r,g\in\OK$. If
\begin{equation}\label{principal}
r+(g)\subseteq\cm_\mathfrak{B},
\end{equation}
then for some $\mathfrak{b}\in\mathfrak{B}$
\[
\gcd((r),(g))=(r)+(g)\subseteq\mathfrak{b}.
\]
\end{Th}
Lemmas needed for the proof of Theorem~\ref{cor_toeplitz} are used also in the proof of Theorem~\ref{Thm_SK}. We decided to keep both results to make the arguments easier to digest. The auxiliary lemmas that we present now are valid in any Dedekind ring $R$ but we phrase them here for $R=\OK$. The proofs of the lemmas are contained in the next subsection.
\begin{Lemma}\label{lem2_multi}
Let $\mathfrak{a}_1,\ldots,\mathfrak{a}_n,\mathfrak{b}\subsetneq\OK$ be pairwise coprime (proper) ideals and $x_1,\ldots,x_n\in\OK$. Then
\[
\mathfrak{b}\not\subseteq \bigcup_{i=1}^n(x_i+\mathfrak{a}_i).
\]
\end{Lemma}

Given an ideal $\mathfrak{a}$ and $y\in\OK$, we set
\[
v_{\mathfrak{a}}(y):=\max\{t : y\in \mathfrak{a}^t\}.
\]
\begin{Lemma}\label{lem3_multi}
Let $r,g\in\OK$ and $(r,g)=\mathfrak{p}_1^{\alpha_1}\cdot\ldots\cdot\mathfrak{p}_m^{\alpha_m}$, where $\mathfrak{p}_i$, $1\leq i\leq m$, are distinct prime ideals ordered in such a way that for some $0\leq m'\leq m$ we have
\begin{equation}\label{c3}
\nu_{\mathfrak{p}_i}(g)=\nu_{\mathfrak{p}_i}(r) \iff i>m'.
\end{equation}
Then for
\[
x\in\mathfrak{p}_{m'+1}\cdot\ldots\cdot\mathfrak{p}_m\setminus\bigcup_{i=1}^{m'}\mathfrak{p}_i,
\]
we have
\[(r+xg)=(r,g)\mathfrak{q}_1\ldots\mathfrak{q}_n,\]
where $\mathfrak{q}_i$ are prime ideals and $\mathfrak{q}_i\neq\mathfrak{p}_j$ for all $1\leq i\leq n$ and all $1\leq j\leq m$.
\end{Lemma}

We will need the following generalization of Lemma \ref{lem3_multi}.
\begin{Lemma}\label{lem4_multi}
Let $r,g\in\OK$. Then there exists a sequence $x_1,x_2,\ldots\in\OK$ such that
\[
(r+x_ig)=(r,g)\mathfrak{q}_1^{(i)}\ldots\mathfrak{q}_{m_i}^{(i)},
\]
where $\mathfrak{q}_k^{(i)}$ are prime ideals and $\mathfrak{q}_k^{(i)}\neq\mathfrak{q}_l^{(j)}$ for $i\neq j$, $1\leq k\leq m_i$, $1\leq l\leq m_j$ and $\mathfrak{q}_k^{(i)}\ndivides(r,g)$ for $i\geq1$. In particular, for $i\neq j$, we have
\[(r+x_ig,r+x_jg)=(r,g).\]
\end{Lemma}
Note that Lemma~\ref{lem4_multi} will play the same role in the proof of Theorem~\ref{cor_toeplitz} as Lemma~\ref{lem1} in the proof of Theorem~\ref{druggi2Toeplitz}. To see that Lemma~\ref{lem4_multi} is indeed a multidimensional version of Lemma~\ref{lem1} notice that we can rephrase Lemma~\ref{lem1}.
\begin{Lemma}
Let $d,r\in\N$ Then there exists a sequence $(n_i)_{i\geq 1}\subseteq \Z$ and an infinite pairwise coprime set $\mathcal{C}=\{c_1,c_2,\dots\}$ such that
\[
r+dn_i=\gcd(r,d)c_i \text{ for }i\geq 1.
\]
\end{Lemma}

\begin{proof}[Proof of Theorem~\ref{cor_toeplitz}]
Let $x_1,x_2,\ldots\in\OK$ be the sequence from Lemma \ref{lem4_multi}. For $i\geq1$, let $\mathfrak{b}_i\in\mathfrak{B}$ be such that
\[
r+x_ig\in\mathfrak{b}_i.
\]
By Lemma~\ref{lem4_multi}, we have \[(r,g)\mathfrak{q}_1^{(i)}\ldots\mathfrak{q}_{m_i}^{(i)}=(r+x_ig)\subseteq\mathfrak{b}_i.\] We claim that there exists $\mathfrak{b}\in\mathfrak{B}$ such that $\mathfrak{b}_i=\mathfrak{b}$ for infinitely many $i\geq1$. Indeed, if this is not the case, then passing to a subsequence, we can assume that all $\mathfrak{b}_i$ are distinct. Moreover, since $(r,g)$ has finitely many factors, passing to a further subsequence if necessary, we have that $\gcd(\mathfrak{b}_i,(r,g))=\mathfrak{d}$ for some ideal $\mathfrak{d}$. Now, $\mathfrak{b}_i=\gcd(\mathfrak{b}_i,(r,g))\mathfrak{c}_i'$, where $\mathfrak{c}_i\divides \mathfrak{q}_1^{(i)}\ldots\mathfrak{q}_{m_i}^{(i)}$. It follows that $\mathfrak{d}\in\mathfrak{D}$, which is a contradiction.

Consider the prime factorization of $\mathfrak{b}$:
\[
\mathfrak{b}=\mathfrak{q}'_1\cdot\ldots\cdot\mathfrak{q}'_n
\]
and take $i$ such that
\begin{equation}\label{f21}
\{\mathfrak{q}^{(i)}_1,\dots,\mathfrak{q}^{(i)}_{m_i}\}\cap \{\mathfrak{q}'_1,\dots,\mathfrak{q}'_n\}=\emptyset.
\end{equation}
We claim that this yields $(r,\mathfrak{a})\subseteq \mathfrak{b}$. Indeed,
\[
r+x_{i}g\in\mathfrak{b} \iff \mathfrak{b} \divides (r+x_{i}g)=(r,g)\mathfrak{q}_1^{(i)}\cdot\ldots\cdot \mathfrak{q}_{m_i}^{(i)},
\]
which, in view of~\eqref{f21}, implies that $\mathfrak{b}\divides (r,g_1,\dots,g_t)$, i.e.\ $(r,\mathfrak{a})\subseteq \mathfrak{b}$.
%
%
%
\end{proof}

To prove Theorem~\ref{Thm_SK}, we will need the following extension of Lemma~\ref{lem4_multi}.
\begin{Lemma}\label{lem5_multi}
Let $r,g_1,\ldots,g_t\in\OK$ and $(r,g_1,\ldots,g_t)=\mathfrak{p}_1\ldots\mathfrak{p}_m$. Then there exists a sequence  \[
\underline{x}_1=(x_{1,1},\ldots,x_{1,t}),\underline{x}_2=(x_{21},\ldots,x_{2t}),\ldots\in\OK^t
\]
such that
\[
(r+x_{i,1}g_1+\ldots+x_{i,t}g_t)=(r,g_1,\ldots,g_t)\mathfrak{q}_1^{(i)}\ldots\mathfrak{q}_{m_i}^{(i)},
\]
where $\mathfrak{q}_j^{(i)}$ are prime ideals and
$\mathfrak{q}_k^{(i)}\neq\mathfrak{q}_l^{(j)}$ for $i\neq j$, $1\leq k\leq m_i$, $1\leq l\leq m_j$ and $\mathfrak{q}_k^{(i)}\ndivides(r,a_1,\ldots,a_t)$ for $i\geq1$.
\end{Lemma}

\begin{proof}[Proof of Theorem \ref{Thm_SK}]
Let $g_1,\dots,g_t$ be a set of generators\footnote{It is known that every ideal in a Dedekind ring is generated by 2 elements, so we could assume that $t\le 2$. This, however, does not simplify the proofs substantially.}    of $\mathfrak{a}$. Let $\underline{x}_1,\underline{x}_2,\dots$ be as in Lemma~\ref{lem5_multi}. Since $\mathfrak{B}$ does not contain a rescaled copy of an infinite pairwise coprime set, we have
\[
r+x_{i,1}g_1+\dots + x_{i,t}g_t\in\mathfrak{b}
\]
for infinitely many $i$'s for some $\mathfrak{b}\in\mathfrak{B}$. Consider the prime factorization of $\mathfrak{b}$:
\[
\mathfrak{b}=\mathfrak{q}'_1\cdot\ldots\cdot\mathfrak{q}'_n
\]
and take $i$ such that
\begin{equation}\label{f2}
\{\mathfrak{q}^{(i)}_1,\dots,\mathfrak{q}^{(i)}_{m_i}\}\cap \{\mathfrak{q}'_1,\dots,\mathfrak{q}'_n\}=\emptyset.
\end{equation}
We claim that this yields $(r,\mathfrak{a})\subseteq \mathfrak{b}$. Indeed,
\[
r+x_{i,1}g_1+\dots+x_{i,t}g_t\in\mathfrak{b} \iff \mathfrak{b} \divides (r+x_{i,1}g_1+\dots+x_{i,t}g_t)=(r,g_1,\dots,g_t)\mathfrak{q}_1^{(i)}\cdot\ldots\cdot \mathfrak{q}_{m_i}^{(i)},
\]
which, in view of~\eqref{f2}, implies that $\mathfrak{b}\divides (r,g_1,\dots,g_t)$, i.e.\ $(r,\mathfrak{a})\subseteq \mathfrak{b}$.
%
\end{proof}

\paragraph{Proofs of lemmas}
\begin{proof}[Proof of Lemma~\ref{lem2_multi}]
This is a simple consequence of the Chinese Remainder Theorem for rings. Indeed, for any choice of $y_1,\dots, y_n\in\OK$, one can find $x\in \OK$ such that
\[
x\in\mathfrak{b} \text{ and }x-y_i \in\mathfrak{a}_i \text{ for }1\leq i\leq n.
\]
Now, it suffices to apply the above to any $y_1,\dots, y_n$ such that $y_i-x_i\not\in\mathfrak{a}_i$ to conclude that
\[
\mathfrak{b}\cap \bigcap_{i=1}^{n}(x_i+\mathfrak{a}_i)^c\neq\emptyset.
\]
\end{proof}
\begin{proof}[Proof of Lemma~\ref{lem3_multi}]
For every $1\leq i\leq m$, we will prove that $v_{\mathfrak{p}_i}(r+xg)=\alpha_i$ for each $1\leq i\leq m$. Consider first $i>m'$. We have
\[
v_{\mathfrak{p}_i}(r)=v_{\mathfrak{p}_i}(g)=\alpha_i  \text{ and } v_{\mathfrak{p}_i}(x)=1.
\]
It follows that, as $v_{\mathfrak{p}_i}(r)=\alpha_i<v_{\mathfrak{p}_i}(xg)$,  $v_{\mathfrak{p}_i}(r+xg)= \alpha_i$.

Now, consider $1\leq i\leq m'$. We have
\[
v_{\mathfrak{p}_i}(x)=0, v_{\mathfrak{p}_i}(r)\ge \alpha_i, v_{\mathfrak{p}_i}(g)\ge \alpha_i
\]
and precisely one of the following:
\begin{equation}\label{gie1}
\text{either } v_{\mathfrak{p}_i}(r)>\alpha_i \text{ or }v_{\mathfrak{p}_i}(g)>\alpha_i.
\end{equation}
It follows that $v_{\mathfrak{p}_i}(r)\neq v_{\mathfrak{p}_i}(xg)$, thus $v_{\mathfrak{p}_i}(r+xg)=\min(v_{\mathfrak{p}_i}(r),v_{\mathfrak{p}_i}(xg))=\alpha_i$.

This yields $v_{\mathfrak{p}_i}(r+xg)=\alpha_i$ for each $1\leq i\leq m$.
\end{proof}

\begin{proof}[Proof of Lemma~\ref{lem4_multi}]
Let $\mathfrak{p}_1,\dots,\mathfrak{p}_m$, $\alpha_1,\ldots,\alpha_m$ and $m'$ be as in the assumptions of Lemma~\ref{lem3_multi} and let
\[
x_1\in\mathfrak{p}_{m'+1}\cdot\ldots\cdot\mathfrak{p}_m\setminus (\bigcup_{i=1}^{m'}\mathfrak{p}_i).
\]
Then
\[
(r+x_1g)=\mathfrak{p}_1^{\alpha_1}\cdot\ldots\cdot\mathfrak{p}_m^{\alpha_m}\mathfrak{q}_1^{(1)}\cdot\ldots\cdot\mathfrak{q}_{m_1}^{(1)}
\]
and $\mathfrak{q}_k^{(1)}\not\in\{\mathfrak{p}_1,\dots,\mathfrak{p}_m\}$ are prime ideals. Suppose that we have chosen $x_1,\dots, x_{i-1}$ and take
\[
x_i\in\mathfrak{p}_{m'+1}\cdot\ldots\cdot\mathfrak{p}_m\setminus \left(\bigcup_{j=1}^{m'}\mathfrak{p}_j\cup\bigcup_{j=1}^{i-1}\bigcup_{l=1}^{m_j}(\mathfrak{q}_l^{(j)}+x_j)\right)
\]
(Lemma~\ref{lem2_multi} guarantees that such $x_i$ exists). It follows by Lemma~\ref{lem3_multi} that
\[
(r+x_ig)=\mathfrak{p}_1^{\alpha_1}\cdot\ldots\cdot\mathfrak{p}_m^{\alpha_m}\cdot\mathfrak{q}_1^{(k)}\cdot\ldots\cdot\mathfrak{q}_{m_k}^{(k)},
\]
where $\mathfrak{q}_k^{(i)}\not\in\{\mathfrak{p}_1,\dots,\mathfrak{p}_m\}$.

It remains to show that
$\mathfrak{q}_k^{(i)}\neq \mathfrak{q}_\ell^{(j)}$ whenever $j\neq i$. Suppose otherwise, then
\[
(x_j-x_k)g\in\mathfrak{q}_\ell^{(j)}
\]
Now, $x_j-x_i\notin \mathfrak{q}_\ell^{(j)}$, as $x_i\not\in \mathfrak{q}_\ell^{(j)}+x_j$. Thus, $g\in \mathfrak{q}_\ell^{(j)}$ and, consequently, $r=r+x_ig-x_ig \in\mathfrak{q}_\ell^{(j)}$. This is impossible, as $\mathfrak{q}_\ell^{(j)}\neq \mathfrak{p}_i$ for all $i$. The remaining statement follows easily.

\end{proof}

\begin{proof}[Proof of Lemma~\ref{lem5_multi}]
The proof is inductive. For $t=1$ the assertion becomes the one of Lemma~\ref{lem4_multi}. Now, take $t>1$. There exists a sequence
\[
\underline{x}_1'=(x'_{1,1},\dots,x'_{1,t-1}),\underline{x}'_2=(x'_{2,1},\dots,x'_{2,t-1}),\dots \in \OK^{t-1}
\]
such that
\[
(r+x'_{i,1}g_1+\dots+x'_{i,t-1}g_{t-1})=(r,g_1,\dots,g_{t-1})\cdot \mathfrak{s}_1^{(i)}\cdot\ldots\cdot\mathfrak{s}_{m'_i}^{(i)},
\]
where $\mathfrak{s}_k^{(i)}$ are prime ideals such that $\mathfrak{s}_k^{(i)}\neq \mathfrak{s}_l^{(j)}$, whenever $i\neq j$. Fix
\[
\underline{x}'=(x'_1,\dots,x'_{t-1})
\]
such that
\[
(r+x'_1g_1+\dots+x'_{t-1}g_{t-1})=(r,g_1,\dots,g_{t-1})\mathfrak{s}_1\cdot\ldots\cdot\mathfrak{s}_{m'}
\]
and $g_t\not\in\bigcup_{j=1}^{m'}\mathfrak{s}_j$. By Lemma~\ref{lem4_multi} (applied to $r'=r+x_1'g_1+\dots +x_{t-1}'g_{t-1}$ and $g'=g_t$), there exists a sequence $(x_i)_{i\geq 1}\subseteq \OK$ such that, for every $i\ge 1$:
\begin{equation}\label{f1}
(r'+x_ig')=(r',g')\mathfrak{c}_1^{(i)}\cdot\ldots\cdot \mathfrak{c}^{(i)}_{l_i},
\end{equation}
i.e.\
\[
(r+x_1'g_1+\dots+x'_{t-1}g_{t-1}+x_ig_t)=(r+x'_1g_1+\dots+x'_{t-1}g_{t-1},g_t)\mathfrak{c}_1^{(i)}\cdot\ldots\cdot\mathfrak{c}^{(i)}_{l_i},
\]
where $\mathfrak{c}^{(i)}_1,\dots,\mathfrak{c}^{(i)}_{l_i}$ are prime ideals that do not divide $(r',g')$ and
\[
\{\mathfrak{c}^{(i)}_1,\dots\mathfrak{c}^{(i)}_{l_i}\}\cap\{\mathfrak{c}^{(j)}_1,\dots\mathfrak{c}^{(j)}_{l_j}\}=\emptyset \text{ whenever }i\neq j.
\]
Now,
\[
(r',g')=(r+x'_1g_1+\dots+x'_{t-1}g_{t-1},g_t)=(r,g_1,\dots,g_{t-1})\mathfrak{s}_1\cdot\ldots\cdot\mathfrak{s}_{m'}+(g_t).
\]
Clearly,  $(r',g')\subseteq (r,g_1,\dots,g_{t})$. On the other hand, 
since $g_t\notin \mathfrak{s}_1\cup\ldots\cup\mathfrak{s}_{m'}$, it follows that  $1\in (g_t)+\mathfrak{s}_1\cdot\d\ldots\cdot\mathfrak{s}_{m'}$. Thus   
\[
(r,g_1,\dots,g_{t-1})+(g_t)\subseteq (r,g_1,\dots,g_{t-1})\mathfrak{s}_1\cdot\ldots\cdot\mathfrak{s}_{m'}+(g_t),
\]
and we get
\[
(r',g')=(r,g_1,\dots,g_t).
\]
This equality, combined with~\eqref{f1}, yields
\[
(r+x_1'g_1+\dots+x'_{t-1}g_{t-1}+x_ig_t)=(r,g_1,\dots,g_t)\mathfrak{c}_1^{(i)}\cdot\ldots\cdot\mathfrak{c}_{l_i}^{(i)}.
\]
\end{proof}
\subsection{Non-periodic positions and periodic structure}
The first goal of this section is to give a description of non-periodic positions on $\eta$, i.e.\ to prove Proposition~\ref{propA}.

We have
\[
\mathcal{F}_{\mathfrak{B}^*}=\mathcal{F}_\mathfrak{B}\cap \mathcal{F}_\mathfrak{D},
\]
where $\mathfrak{D}$ is defined in~\eqref{A}. Equivalently, this equality can be rewritten as
\begin{equation}\label{eta*}
\eta^*_n=\begin{cases}
\eta_n \text{, if }n\not\in\mathcal{M}_{\mathfrak{D}},\\
0 \text{, if }n\in\mathcal{M}_\mathfrak{D},
\end{cases}
\end{equation}
Recall that for any ideal $\mathfrak{s}\subseteq\OK$ places which are $\mathfrak{s}$-periodic are denoted by
\[\Per(\eta,\mathfrak{s})=\{g\in\OK;\ \eta_g=\eta_{g+s}\text{ for any }s\in\mathfrak{s}\}.\]
Let $(\mathfrak{S}_i)_{i\geq1}$ be a saturated filtration of $\mathfrak{B}$ by finite collections of ideals. For $i\geq1$ put
\[
\mathfrak{s}_i:=\bigcap_{\mathfrak{b}\in\mathfrak{S}_i}\mathfrak{b}=\lcm(\mathfrak{S}_i).
\]
\begin{Lemma}\label{ideal_sol}
Let $g\in\G$ and~$\{e_\G\}\neq\HH_1,\HH_2\subseteq\G$ be subgroups. Then
\begin{align}\label{kogruencja}
(\HH_1+g)\cap\HH_2\neq\emptyset\text{ if and only if }g\in\HH_1+\HH_2.
\end{align}
\end{Lemma}
The proof of the above lemma is straightforward, so we skip it.
%
\begin{Lemma}[see i.e.\ {\cite[Chapter I, Proposition 2.2]{MR1878556}}]\label{intersect}
Let $\HH_1,\HH_2\subseteq\G$ be subgroups of finite index. Then $\HH_1\cap\HH_2$ is also a subgroup of finite index in~$\G$ and \[[\G : \HH_1\cap\HH_2]=[\G : \HH_1]\cdot[\HH_1 : \HH_1\cap\HH_2].\]
\end{Lemma}
\begin{Lemma}[{\cite[2.4. s. 128]{MR36767}}]\label{index}
If $\G$ is a finitely generated group, then it contains finitely many (possibly $0$) subgroups of a given index $n$.
\end{Lemma}

The proof of Proposition~\ref{propA} go along the same lines as the proof of Proposition~\ref{lemmaA_one}.

\begin{proof}[Proof of Proposition~\ref{propA}] 
It suffices to show that
\begin{equation}\label{dwie_inkluzje_ok}
\OK\setminus\bigcup_{i\geq1}\Per(\eta,\mathfrak{s}_i)\subseteq\mathcal{M}_{\mathfrak{D}}\cap \mathcal{F}_{\mathfrak{B}}\subseteq\OK\setminus \bigcup_{\mathfrak{s}\subset \OK}\Per(\eta,\mathfrak{s}).
\end{equation}
Clearly, $\OK\setminus \bigcup_{\mathfrak{s}\subset\OK}\Per(\eta,\mathfrak{s})\subseteq \OK\setminus \bigcup_{i\geq 1}\Per(\eta,\mathfrak{s}_i)$ and therefore the inclusions must be, in fact, equalities.

\paragraph{First inclusion.} Take $n\not\in \bigcup_{i\geq 1}\Per(\eta,\mathfrak{s}_i)$. Then $n\in\mathcal{F}_\mathfrak{B}$. Moreover, $n+\mathfrak{s}_k\not\subseteq \mathcal{F}_\mathfrak{B}$ for any $k\geq 1$. In other words, $(n+\mathfrak{s}_k)\cap\mathcal{M}_\mathfrak{B}\neq\emptyset$, so there exists $\mathfrak{b}_k$ such that $n\in \mathfrak{b}_k+\mathfrak{s}_k=\gcd(\mathfrak{b}_k,\mathfrak{s}_k)$, so that
\[
n\OK \subseteq \mathfrak{b}_k+\mathfrak{s}_k=\gcd(\mathfrak{b}_k,\mathfrak{s}_k).
\]
Notice that $b_k\not\in\mathfrak{S}_k$ (if $b_k\in\mathfrak{S}_k$ then $\mathfrak{s}_k\subseteq\mathfrak{b}_k$ and $n\in\mathfrak{b}_k+\mathfrak{s}_k\subseteq \mathfrak{b}_k$ which contradicts $n\in\mathcal{F}_\mathfrak{B}$). It follows by Lemma~\ref{intersect} that the index of $\mathfrak{b}_k+\mathfrak{s}_k$ is divisible by the index of $n\OK$, so by Lemma~\ref{index},  there are only finitely many possibilities for $\mathfrak{b}_k+\mathfrak{s}_k$. In other words, there exists an ideal $\mathfrak{d}$ and a subsequence $(k_\ell)\subseteq\N$ such that $\gcd(\mathfrak{b}_{k_\ell},\mathfrak{s}_{k_\ell})=\mathfrak{b}_{k_\ell}+\mathfrak{s}_{k_\ell}=\mathfrak{d}$ (and $n\in\mathfrak{d}$, as $n\in \mathfrak{b}_k+\mathfrak{s}_k$ for each $k\geq 1$).

Let $\ell_2>\ell_1:=1$ be such that
\begin{equation}\label{rnie}
\mathfrak{b}_{k_{\ell_1}}\in\mathfrak{S}_{k_{\ell_2}} \text{ (in particular, $\mathfrak{s}_{k_{\ell_2}}\subseteq \mathfrak{b}_{k_{\ell_1}}$).}
\end{equation}
By the choice of $(k_\ell)$, we have $\mathfrak{b}_{k_{\ell_1}},\mathfrak{b}_{k_{\ell_2}}\subseteq \mathfrak{d}$, so
\[
\mathfrak{b}_{k_{\ell_1}}+\mathfrak{b}_{k_{\ell_2}}\subseteq \mathfrak{d}.
\]
On the other hand, using~\eqref{rnie},  we obtain
\[
\mathfrak{d}=\mathfrak{b}_{k_{\ell_2}}+\mathfrak{s}_{k_{\ell_2}}=\gcd(\mathfrak{b}_{k_{\ell_2}},\mathfrak{s}_{k_{\ell_2}})\subseteq\gcd(\mathfrak{b}_{k_{\ell_2}},\mathfrak{b}_{k_{\ell_1}}).
\]
Thus, $\gcd(\mathfrak{b}_{k_{\ell_1}},\mathfrak{b}_{k_{\ell_2}})=\mathfrak{b}_{k_{\ell_1}}+\mathfrak{b}_{k_{\ell_2}}=\mathfrak{d}$.

Suppose that we have constructed $1=\ell_1<\ell_2<\ldots<\ell_m$ such that
\begin{equation}\label{nowe1}
\gcd(\mathfrak{b}_{k_{\ell_i}},\mathfrak{b}_{k_{\ell_{i'}}})=\mathfrak{d} \text{ for any }1\leq i<i'\leq m.
\end{equation}
There exists $\ell_{m+1}>\ell_m$ such that
\begin{equation}\label{nowe2}
\mathfrak{b}_{k_{\ell_1}},\mathfrak{b}_{k_{\ell_2}},\dots,\mathfrak{b}_{k_{\ell_m}}\in \mathfrak{S}_{k_{\ell_{m+1}}}.
\end{equation}
It follows by the choice of $(k_\ell)$ that $b_{k_{\ell_i}}\subseteq \mathfrak{d}$ for $1\leq i\leq m+1$. This gives
\[
\gcd(\mathfrak{b}_{k_{\ell_i}},\mathfrak{b}_{k_{\ell_{i'}}})=\mathfrak{b}_{k_{\ell_i}}+\mathfrak{b}_{k_{\ell_{i'}}}\subseteq\mathfrak{d} \text{ for any }1\leq i<i'\leq m+1.
\]
On the other hand, using~\eqref{nowe2}, we obtain $\mathfrak{d}=\gcd(\mathfrak{b}_{k_{\ell_{m+1}}},\mathfrak{s}_{k_{\ell_{m+1}}})\subseteq\gcd(\mathfrak{b}_{k_{\ell_i}},\mathfrak{b}_{k_{\ell_{m+1}}})$.
Thus,
\[
\gcd(\mathfrak{b}_{k_{\ell_i}},\mathfrak{b}_{k_{\ell_{i'}}})=\mathfrak{d} \text{ for any }1\leq i<i'\leq m+1
\]
and the above inductive procedure therefore yields a sequence $1=\ell_1<\ell_2<\dots$ such that
\[
\gcd(\mathfrak{b}_{k_{\ell_i}},\mathfrak{b}_{k_{\ell_{i'}}})=\mathfrak{d} \text{ for any }i\neq i'.
\]
Let $\mathfrak{c}_i$, for $i\geq 1$, be an ideal such that $\mathfrak{c}_i\mathfrak{d}=\mathfrak{b}_{k_{\ell_i}}$. It remains to show that $\mathcal{C}:=\{\mathfrak{c}_i : i\geq 1\}$ is pairwise coprime. However, we have
\[
\mathfrak{d}=\gcd(\mathfrak{b}_{k_{\ell_i}},\mathfrak{b}_{k_{\ell_{i'}}})=\mathfrak{b}_{k_{\ell_i}}+\mathfrak{b}_{k_{\ell_{i'}}}=(\mathfrak{c}_i+\mathfrak{c}_{i'})\mathfrak{d}
\]
and after dividing the above by $\mathfrak{d}$, we conclude that $\mathfrak{c}_i+\mathfrak{c}_{i'}=\OK$ for $i\neq i'$. This tells us that $n\in\mathcal{M}_\mathfrak{D}$.

\paragraph{Second inclusion.} Take $n\in\mathcal{M}_\mathfrak{D}\cap\mathcal{F}_\mathfrak{B}$.  There exists a non-zero ideal $\mathfrak{d}$ and an infinite pairwise coprime collection $\mathcal{C}$ of ideals such that
\[
n\in\mathfrak{d} \text{ and }\mathfrak{d}\mathcal{C}\subseteq \mathfrak{B}.
\]
Suppose that $n\in\Per(\eta,\mathfrak{s})$ for some non-zero ideal $\mathfrak{s}$. In other words, we have $n+\mathfrak{s}\subseteq \mathcal{F}_\mathfrak{B}$. Since $n\in\mathfrak{d}$, it follows that
\begin{equation}\label{moje}
\mathfrak{d}\cap (n+\mathfrak{s})=n+\mathfrak{d}\cap\mathfrak{s}\subseteq \mathcal{F}_\mathfrak{B}.
\end{equation}
Take $\mathfrak{c}\in\mathcal{C}$ coprime to $\mathfrak{s}$ and $\mathfrak{d}$. Then
\[
n\in\mathfrak{d}=\mathfrak{d}(\mathfrak{s}+\mathfrak{c})=\mathfrak{d}\mathfrak{s}+\mathfrak{d}\mathfrak{c}=\mathfrak{d}\cap\mathfrak{s}+\mathfrak{d}\mathfrak{c}.
\]
Therefore (by Lemma~\ref{ideal_sol} for $\G=\OK$, $g=n$, $\HH_1=\mathfrak{d}\cap\mathfrak{s}_i$ and~$\HH_2=\mathfrak{d}\mathfrak{c}_j$) we obtain
\[
(n+\mathfrak{d}\cap\mathfrak{s})\cap \mathcal{M}_\mathfrak{B}\supseteq(n+\mathfrak{d}\cap\mathfrak{s})\cap \mathfrak{d}\mathfrak{c}\neq\emptyset,
\]
which contradicts~\eqref{moje}.
\end{proof}

\begin{proof}[Proof of Theorem~\ref{thxC}]
For $i\geq1$ put $\mathfrak{s}_i:=\bigcap_{\mathfrak{b}\in\mathfrak{S}_i}\mathfrak{b}=\lcm(\mathfrak{S}_i)$. By Proposition~\ref{propA}, we get \[\bigcup_{i\geq1}\Per(\eta,\mathfrak{s}_i)=\OK.\]
Clearly, $\mathfrak{s}_{i+1}\subseteq\mathfrak{s}_i$. Suppose that $\mathfrak{s}$ is such that
\[
\Per(\eta,\mathfrak{s}_i)\subseteq\Per(\eta,\mathfrak{s}).
\]
For any $\mathfrak{b}\in\mathfrak{S}_i$, we have
\[
\mathfrak{b}\subseteq \Per(\eta,\mathfrak{s}_i)\subseteq \Per(\eta,\mathfrak{s}),
\]
whence
\[
\mathfrak{b}+\mathfrak{s}\subseteq \mathcal{M}_\mathfrak{B}.
\]
Take $r\in\mathfrak{s}$. It follows by Theorem~\ref{Thm_SK} (for $\mathfrak{a}=\mathfrak{b}$) that for some $\mathfrak{b}'\in\mathfrak{B}$ we have
\[
(r)+\mathfrak{b}\subseteq \mathfrak{b}'.
\]
In particular, $\mathfrak{b}\subseteq\mathfrak{b}'$, which, by the primitivity of $\mathfrak{B}$ implies $\mathfrak{b}'=\mathfrak{b}$. Thus,
\[
\mathfrak{s}\subseteq \mathfrak{b},
\]
as $r\in\mathfrak{s}$ was arbitrary. Now, since $\mathfrak{b}\in\mathfrak{S}_i$ was arbitrary, we obtain $\mathfrak{s}\subseteq \mathfrak{s_i}$
and we conclude that $\mathfrak{s}$ is an essential group of periods.
\end{proof}
\subsection{Minimality}
\subsubsection{Tools}
We begin this section with a multidimensional version of a result from~\cite{MR3803141} used therein to prove \cite[Theorem A]{MR3803141} (i.e.\ Theorem~\ref{minimal_basic}).
\begin{Prop}[{\cite[Proposition 2.15]{MR3803141} for $\G=\Z$}]\label{LE51new}
Let $\G$ be countable abelian group. Let $(X,(T_g)_{g\in\G})$ be a topological dynamical system with a transitive point $x_0\in X$. Then the following conditions are equivalent:
\begin{enumerate}[(a)]
\item\label{HA}
$(X,(T_g)_{g\in\G})$ has a unique minimal subset $M$,
\item\label{HB}
there exists a closed, $(T_g)_{g\in\G}$-invariant subset $M'\subseteq X$ such that
for any $x\in M'$ and~$y\in X$, there exists $(g_n)_{n\geq 1}\subseteq \G$ such that $g_n\to\infty$ and $T_{g_n}y\to x$ as $n\to\infty$,
\item\label{HC}
there exists $x_M\in X$ such that for any $y\in X$ there exists $(g_n)_{n\geq 1}\subseteq \G$ such that $g_n\to\infty$ and $T_{g_n}y\to x_M$,
\item\label{HD}
there exists a closed, $(T_g)_{g\in\G}$-invariant subset $M''\subseteq X$, such that $\left\{g\in\G;\right.\\ \left. T_gx_0 \in U\right\}$ is syndetic for any open $U$ with $U\cap M''\neq\emptyset$,
\item\label{HE}
there exists a sequnce of open sets $(U_n)_{n\geq 1}\subseteq X$ such that:
\begin{itemize}
\item
$\text{diam}(U_n)\to 0\text{ as }n\to\infty$,
\item
the set $\{g\in\G;\ T_gx_0 \in U_n\}$ is syndetic for each $n\in\N$.
\end{itemize}
\end{enumerate}
Furthermore, if any of the above hold, then $M=M'=M''$ and $x_M\in M$.
\end{Prop}
\begin{proof}
\eqref{HA}$\Rightarrow$\eqref{HB} Let $M$ be a unique minimal subset of $(X,(T_g)_{g\in\G})$. Let $y\in X$. Since $X$ is a compact space, so the set \[\omega(y):=\{z\in X; \ \exists g_n\to\infty\text{ such that }T_{g_n}y\to z\}\] is closed, non-empty and~$(T_g)_{g\in\G}$-invariant. We will show that $\bigcap_{y\in X}\omega(y)$ is minimal. Let $x\in\bigcap_{y\in X}\omega(y)$. Then since $\bigcap_{y\in X}\omega(y)$ is closed and $(T_g)_{g\in\G}$-invariant, so $X_x\subseteq\bigcap_{y\in X}\omega(y)$. Obviously, we have $\omega(x)\subseteq X_x$ and $\bigcap_{y\in X}\omega(y)\subseteq\omega(x)$. So $\omega(x)=X_x=\bigcap_{y\in X}\omega(y)$. Hence $\bigcap_{y\in X}\omega(y)$ is minimal. By~\eqref{HA}, we get $M=\bigcap_{y\in X}\omega(y)$. Notice that any set $M'$ satisfying~\eqref{HB} is contained in the set $\bigcap_{y\in X}\omega(y)$, so it also satisfies the equality $M'=M$.

\eqref{HB}$\Rightarrow$\eqref{HC} Take any $x_M\in M'$.

\eqref{HC}$\Rightarrow$\eqref{HD} Let $x_M$ satisfy~\eqref{HC} and put $M'':=X_{x_M}$. Let $U$ be an open set such that $U\cap M''\neq\emptyset$. Let $y\in U\cap M''$. Then there exists $g\in\G$ such that $T_gx_M\in U$. Thus $x_M\in V:=T_{-g}(U)$. From that \[g+\{h\in\G;\ T_hx_0\in V\}= \{h\in\G;\ T_hx_0 \in U\}.\]
If $\{h\in\G ;\ T_hx_0\in V\}$ is not syndetic, then for any $n\in\N$ we have $\{h\in\G ;\ T_hx_0\in V\}+F_n\neq\G$. So for any $n\in\N$ there exists $k_n\in\G$ such that $T_{k_n-f}x_0\in X\setminus V$ for any $f\in F_n$. Let $T_{k_{n_{\ell}}}x_0\to z$ and $g\in\G$. Then $g\in F_{{n_\ell}}$ for sufficiently large $\ell\geq1$. Since \[T_{k_{n_\ell}-g}x_0\in X\setminus V\text{ and }T_{k_{n_\ell}-g}x_0\to T_{-g}z,\] so $T_{-g}z\in X\setminus V$. So $X_z\subseteq X\setminus V$, but then $x_M\not\in X_z$, which contradicts~\eqref{HC}.

\eqref{HD}$\Rightarrow$\eqref{HE} Let $x\in M''$. Take $U_n:=\{y\in X;\ D(x,y)<\frac{1}{n}\}$.

\eqref{HE}$\Rightarrow$\eqref{HA} Assume that~\eqref{HE} holds and $M_1,M_2\subseteq X$ are disjoint minimal sets, so \[\min\{D(x_1,x_2);\ x_i\in M_i, i=1,2\}=:\vep>0.\] From $(U_n)_{n\geq1}$ we take $n$ such that ${\rm diam}\,U_n<\vep/2$. We may assume that $U_n\cap M_2=\emptyset$. Let $V\supseteq M_2$ be open such that $V\cap U_n=\emptyset$. By syndeticity, we get there exists $N\geq1$ such that \begin{equation}\label{synd}
\{g\in\G;\ T_gx_0\in U_n\}+F_N=\G.
\end{equation} Since $M_2$ is $(T_g)_{g\in\G}$-invariant, there is an open set $W$ with $M_2\subseteq W\subseteq V$ such that $T_gW\subseteq V$ for any $g\in F_N$. But the orbit of $x_0$ is dense in $X$; therefore there exists $g'\in\G$ with $T_{g'}x_0\in W$. By \eqref{synd}, there exist $g\in\G$ and $f\in F_N$ such that $T_gx_0\in U_n$ and $g'=g+f$. So \[
U_n\ni T_gx_0=T_{g'-f}x_0=T_{-f}(T_{g'}x_0)\in V,\]
contradicting $U_n\cap V=\emptyset$.

By uniqueness of $M''$ and by above we have $M\subseteq M'\subseteq M''$ and the reduction of $M'$ has no impact on \eqref{HD}, the equalities hold. To finish the proof note that $x_M$ is minimal, so $x_M\in M=M'=M''$.
\end{proof}
\begin{Cor}[{\cite[Corollary 2.17]{MR3803141} for $\G=\Z$}]\label{dlapod}
Let $(X,(S_g)_{g\in\G})$ be a subshift with a transitive point $x_0$. Then $(X,(S_{g})_{g\in\G})$ is essentially minimal if and only if there exists an inifite family of pairwise distinct blocks $\{B_i\}_{i\geq1}$ such that $\{g\in\G;\ {x_0|}_{B_i}={x_0|}_{B_i+g}\}$ is syndetic for any $i\geq1$.
\end{Cor}
\begin{proof}
It follows immediately from the equivalence \eqref{HA} and \eqref{HE} from Proposition~\ref{LE51new}.
\end{proof}

\subsubsection{Main part}
\begin{proof}[Proof of Theorem~\ref{dwyz}]
We will first show (i) $\iff$ (ii). Clearly, if $\mathfrak{D}=\emptyset$ then $\mathfrak{B}=\mathfrak{B}^*$. For the other direction, notice that
\[
\mathfrak{B}=\mathfrak{B}^*=(\mathfrak{B}\setminus \mathcal{M}_\mathfrak{D})\cup \mathfrak{D}^{prim} \text{ implies }\mathfrak{D}^{prim}\subset \mathfrak{B}.
\]
However, by the definition of $\mathfrak{D}$, for any $\mathfrak{d}\in\mathfrak{D}$, some (non-trivial) multiple of $\mathfrak{d}$ is a member of $\mathfrak{B}$, as $\mathfrak{d}\mathcal{C}\subset \mathfrak{B}$ where $\mathcal{C}$ is inifinite (and pairwise coprime). If $\mathfrak{D}\neq\emptyset$ this contradicts the primitivity of $\mathfrak{B}$.

Let us show that $\mathfrak{D}^*=\emptyset$ (i.e.\ $\eta^*$ is Toeplitz, so (i) $\implies$ (iii)). Suppose that there is some $\mathfrak{d}$ and some infinite pairwise coprime set $\cC$ such that
\[
\mathfrak{d}\mathcal{C}\subseteq \mathfrak{B}^*=(\mathfrak{B}\setminus \mathcal{M}_\mathfrak{D})\cup \mathfrak{D}^{prim}.
\]
Then, without loss of generality (taking a smaller but still infinite and pairwise coprime set $\cC$), one of the following holds:
\begin{enumerate}[(A)]
\item
$\mathfrak{d}\mathcal{C}\subseteq \mathfrak{B}\setminus \mathcal{M}_\mathfrak{D}$,
\item
$\mathfrak{d}\mathcal{C}\subseteq \mathfrak{D}^{prim}$.
\end{enumerate}
If (A) holds then $\mathfrak{d}\mathcal{C}\subset \mathfrak{B}$, so $\mathfrak{d}\in \mathfrak{D}$ and we obtain
\[
\mathfrak{d}\mathcal{C}\subset \mathscr{B}\setminus \mathcal{M}_\mathfrak{D}\subset \mathfrak{B}\setminus \mathfrak{d},
\]
which yields a contradiction. Suppose now that (B) holds. Let $\mathcal{C}=\{\mathfrak{c}_1,\mathfrak{c}_2,\dots\}$ and let $\mathcal{A}_i$ for $i\geq 1$ be infinite, pairwise coprime and such that $\mathfrak{d}\mathfrak{c}_i\mathcal{A}_i\subset \mathfrak{B}$. We can choose $\mathfrak{a}_i\in\mathcal{A}_i$, $i\geq 1$ so that $\{\mathfrak{c}_i\mathfrak{a}_i:i\geq 1\}$ is infinite and pairwise coprime (we just use that each $\mathcal{A}_i$ is infinite and pairwise coprime). This gives again $\mathfrak{d}\in \mathfrak{D}$. This contradicts (B) (by the primitivity of $\mathfrak{D}^{prim}$). We conclude that indeed $\mathfrak{D}^*=\emptyset$.

Finally, we will show (iii) $\implies$ (i). If $\eta$ is Toeplitz then $\mathcal{F}_\mathfrak{B}\cap \mathcal{M}_\mathfrak{D}=\emptyset$ since (by Proposition~\ref{propA}) it is the set of non-periodic positions on $\eta$. Therefore, $\mathcal{M}_\mathfrak{D}\subseteq \mathcal{M}_\mathfrak{B}$ and it follows immediately (cf.\ Remark~\ref{B*}) that $\mathcal{M}_{\mathfrak{B}^*}=\mathcal{M}_{\mathfrak{B}}$. This yields $\mathfrak{B}=\mathfrak{B}^*$ by the primitivity of $\mathfrak{B}$ and $\mathfrak{B}^*$.
\end{proof}

\begin{proof}[Proof of Theorem~\ref{minimal2}]
Since $\mathfrak{B}^*=(\mathfrak{B}\cup\mathfrak{D})^{prim}$, we have  $\eta^*\leq\eta$. Moreover, $\eta^*$ is Toeplitz, by Theorem~\ref{dwyz}.

Now, we will prove that $\eta^*\in X_\eta$. Let $(F_k)_{k\geq1}$ be a nested F\o lner sequence. We will show that for any $N\geq 1$,  $\eta^*|_{F_N}$ appears on $\eta$. Fix $N\geq 1$ and let $i\geq1$ be sufficiently large to ensure that
\begin{equation}\label{righti}
 \mathcal{M}_\mathfrak{D}\cap\mathcal{F}_\mathfrak{B}\cap F_N=(\OK\setminus\Per(\eta,\mathfrak{s}_i))\cap F_N.
\end{equation}
Recall that by Proposition~\ref{propA} we have $\mathcal{M}_\mathfrak{D}=\OK\setminus \bigcup_{i\geq 1}\Per(\eta,\mathfrak{s}_i)$), and $\Per(\eta,\mathfrak{s}_i)$ grows, as $i$ grows. Moreover, by Proposition~\ref{propA}, $\mathcal{M}_\mathfrak{D}\cap\mathcal{F}_\mathfrak{B}\cap F_N$ are all non-periodic positions on $\eta$ restricted to $F_N$. Let $t_N$ be the cardinality of $\mathcal{M}_\mathfrak{D}\cap \mathcal{F}_\mathfrak{B}\cap F_N$ and denote the elements of this set by $I_1,\dots, I_{t_N}$. It follows by the definition of $D$ that for any $1\leq j\leq t_N$, there exists a non-zero ideal $\mathfrak{d}_j$ and an infinite pairwise coprime set $\mathcal{C}_j$ such that
\[
I_j\in \mathfrak{d}_j \text{ and }\mathfrak{d}_j\mathcal{C}_j\subseteq \mathfrak{B}.
\]
Let
\[
\mathfrak{L}:=\lcm(\mathfrak{s}_i,\mathfrak{d}_1,\dots,\mathfrak{d}_{t_N}).
\]
Let $\mathfrak{c}_1\in\mathcal{C}_1$ be coprime to $\mathfrak{L}$ (such an ideal exists since $\mathfrak{L}$ has finitely many factors and $\mathcal{C}_1$ is infinite pairwise coprime). Then $I_1\in\mathfrak{d}_1=\mathfrak{L}+\mathfrak{d}_1\mathfrak{c}_1=\gcd(\mathfrak{L},\mathfrak{d}_1\mathfrak{c}_1)$ and therefore we can find $k_1\in\mathfrak{L}$ such that
\[
k_1+I_1\in\mathfrak{d}_1\mathfrak{c}_1.
\]
Let $\mathfrak{c}_2\in\mathcal{C}_2$ be coprime to $\mathfrak{L}$ and $\mathfrak{c}_1$ (again, such an ideal exists because $\mathfrak{L}$ and $\mathfrak{c}_1$ have finitely many factors and $\mathcal{C}_2$ is infinite pairwise coprime). Since  $\mathfrak{L}\subseteq\mathfrak{d}_2$ and $\mathfrak{L},\mathfrak{c}_1$ and~$\mathfrak{c}_2$ are coprime, we have
\[
k_1+I_2\in \mathfrak{L}+\mathfrak{d}_2=\mathfrak{d}_2
\]
Moreover,
\[\mathfrak{d}_2=
\gcd(\mathfrak{L}\mathfrak{c}_1,\mathfrak{d}_2\mathfrak{c}_2)=\gcd(\mathfrak{L}\cap\mathfrak{c}_1,\mathfrak{d}_2\mathfrak{c}_2)=\mathfrak{L}\cap\mathfrak{c}_1+\mathfrak{d}_2\mathfrak{c}_2
\]
Therefore, $k_1+I_2\in \mathfrak{L}\cap\mathfrak{c}_1+\mathfrak{d}_2\mathfrak{c}_2$ and thus, we can find $k_2\in \mathfrak{L}\cap \mathfrak{c}_1$ such that
\[
k_2+k_1+I_2\in \mathfrak{d}_2\mathfrak{c}_2.
\]
Notice that since $k_2\in\mathfrak{L}\cap\mathfrak{c}_1\subseteq\mathfrak{d}_1\mathfrak{c}_1$, we have
\[
k_2+k_1+I_1\in\mathfrak{d}_1\mathfrak{c}_1.
\]
Let $\mathfrak{c}_3\in\mathcal{C}_3$ be coprime to $\mathfrak{L}$, $\mathfrak{c}_1$ and~$\mathfrak{c}_2$. Since
$\mathfrak{L}\subseteq \mathfrak{d}_3$ and $\mathfrak{L},\mathfrak{c}_1,\mathfrak{c}_2$ and $\mathfrak{c}_3$ are pairwise coprime, we have
\[
k_2+k_1+I_3\in \mathfrak{L}\cap \mathfrak{c}_1 + \mathfrak{L}+\mathfrak{d}_3=\mathfrak{d}_3
\]
Moreover,
\[
\mathfrak{d}_3=\mathfrak{L}\cap\mathfrak{c}_1\cap\mathfrak{c}_2+\mathfrak{d}_3\mathfrak{c}_3.
\]
Therefore, $k_2+k_1+I_3\in \mathfrak{L}\cap\mathfrak{c}_1\cap\mathfrak{c}_2+\mathfrak{d}_3\mathfrak{c}_3$ and thus, we can find $k_3\in\mathfrak{L}\cap\mathfrak{c}_1\cap\mathfrak{c}_2$ such that
\[
k_3+k_2+k_1+I_3\in\mathfrak{d}_3\mathfrak{c}_3.
\]
Notice that since $k_3\in\mathfrak{d}_\ell\mathfrak{c}_\ell$ for $\ell=1,2$, we have
\[
k_3+k_2+k_1+I_\ell\in\mathfrak{d}_\ell\mathfrak{c}_\ell\text{ for }\ell=1,2.
\]
By repeating the above arguments, we obtain $k_j\in\mathfrak{L}\cap\bigcap_{i=1}^{j-1}\mathfrak{c}_i$ and $\mathfrak{c}_j\in\mathcal{C}_j$, $j=1,2,\ldots,t_N$, such that
\begin{equation}\label{31}
k_{t_N}+\ldots+k_2+k_1+I_k\in\mathfrak{d}_j\mathfrak{c}_j
\end{equation}
for $j=1,2,\ldots t_N$. Put
\[
M_N:=k_{t_N}+\ldots+k_2+k_1\in\mathfrak{L}.
\]

We will show that
\begin{equation}\label{32}
\eta_{n+M_N}=\eta_n=\eta^*_n \text{ for }n\in\Per(\eta,\mathfrak{s}_i)\cap F_N.
\end{equation}
Take $n\in\Per(\eta,\mathfrak{s}_i)\cap F_N$. Since $M_N\in \mathfrak{L} \subseteq \mathfrak{s}_i$, it follows that
\begin{equation}\label{33}
\eta_{n+M_N}=\eta_n \text{ for }n\in\Per(\eta,\mathfrak{s}_i)\cap F_N
\end{equation}
and $n+M_N\in\Per(\eta,\mathfrak{s}_i)$. It follows by Proposition~\ref{propA} that $n,n+M_N\not\in\mathcal{M}_\mathfrak{D}\cap\mathcal{F}_\mathfrak{B}$ (the latter set is the set of all non-periodic positions on $\eta$). In other words,
\[
n,n+M_N\in\mathcal{F}_\mathfrak{D}\cup\mathcal{M}_\mathfrak{B}=(\mathcal{F}_\mathfrak{D}\cap\mathcal{F}_\mathfrak{B})\cup\mathcal{M}_\mathfrak{B}.
\]
Recall that we have $\mathfrak{B}^*=(\mathfrak{B}\cup\mathfrak{D})^{prim}$, so $\mathcal{M}_{\mathfrak{B}^*}=\mathcal{M}_\mathfrak{B}\cup\mathcal{M}_\mathfrak{D}$. Equivalently, $\mathcal{F}_{\mathfrak{B}^*}=\mathcal{F}_\mathfrak{B}\cap\mathcal{F}_\mathfrak{D}$. Therefore,
\begin{itemize}
\item if $n\in\mathcal{F}_\mathfrak{D}\cap\mathcal{F}_\mathfrak{B}$ then $n\in\mathcal{F}_{\mathfrak{B}^*}$,
\item if $n\in\mathcal{M}_\mathfrak{B}$ then $n\in\mathcal{M}_{\mathfrak{B}^*}$.
\end{itemize}
Hence,~\eqref{32} indeed holds.

Now, we will show that
\begin{equation}\label{34}
\eta_{n+M_N}=\eta^*_n \text{ for }n\in (\OK\setminus \Per(\eta,\mathfrak{s}_i))\cap F_N.
\end{equation}
Since $\mathfrak{d}_j\mathfrak{c}_j\in\mathfrak{B}$ and $I_j+M_N \in \mathfrak{d}_j\mathfrak{c}_k$, it follows by~\eqref{31} that
\begin{equation}\label{35}
\eta_{I_j+M_N}=0 \text{ for all } 1\leq j\leq t_N.
\end{equation}
Moreover, it follows by $I_j\in \mathcal{M}_{\mathfrak{D}}\subseteq \mathcal{M}_{\mathfrak{B}^*}$ that we also have
\[
\eta^*_{I_j}=0 \text{ for all }1\leq j\leq t_N.
\]
Therefore,~\eqref{34} indeed holds. Combining~\eqref{32} and~\eqref{34}, we conclude that
\begin{equation}\label{36}
\eta|_{F_N+M_N}=\eta^*|_{F_N}.
\end{equation}

Notice that in the above arguments we used only that
\[
M_N\in \mathfrak{L} \text{ and }M_N+I_j\in\mathfrak{d}_j\mathfrak{c}_j
\]
(to obtain~\eqref{33} and~\eqref{35}, respectively). Thus, by~\eqref{31},
\[
\eta|_{F_N+M_N}=\eta|_{F_N+M_N+s}\text{ for any } s\in\mathfrak{L}\cap\bigcap_{j=1}^{s_N}\mathfrak{c}_j.
\]
By Lemma~\ref{intersect}, $\mathfrak{L}\cap\bigcap_{j=1}^{s_N}\mathfrak{c}_j$ has finite index, so $\{s\in\OK;\ \eta|_{F_N+M_N}=\eta|_{F_N+M_N+s}\}$ is syndetic. Moreover, since $(F_k)_{k\geq1}$ is increasing, so blocks $(\eta|_{F_N+M_N})_{N\geq1}$ have different lengths, so they are pairwise different.  By Corollary~\ref{dlapod}, $X_\eta$ has a unique minimal subset.
\end{proof}
\begin{Cor}\label{proxety}
The pair $(\eta,\eta^*)$ is proximal.
\end{Cor}
\begin{proof}
Let $N\geq1$ and let $F_N$, $M_N$, $\mathfrak{s}_i$, $I_j$, $\mathfrak{d}_j$, $\mathcal{L}$, $s_N$ be the same as in the proof of Theorem~\ref{minimal2}.
Then
\begin{equation}\label{etyprox}
\eta|_{F_N+M_N}=\eta^*|_{F_N+M_N}.
\end{equation}
Indeed, since $n+M_N\not\in \mathcal{M}_\mathfrak{D}\cap\mathcal{F}_\mathfrak{B}$, we can show that $\eta_{n+M_N}=\eta^*_{n+M_N}$ for $n\in\Per(\eta,\mathfrak{s}_i)\cap F_N$, using the same arguments as were used in the proof of Theorem~\ref{minimal2} to show $\eta_n=\eta^*_n$ for $n\in\Per(\eta,\mathfrak{s}_i)\cap F_N$

so the same lines as we showed in the proof of Theorem \ref{minimal2} $\eta_n=\eta^*_n$ for $n\in\Per(\eta,\mathfrak{s}_i)\cap F_N$, one can show $\eta_{n+M_N}=\eta^*_{n+M_N}$ for $n\in\Per(\eta,\mathfrak{s}_i)\cap F_N$. Since $I_j\in\mathfrak{d}_j\subseteq\mathcal{M}_\mathfrak{D}$ and $I_j+M_N\in\mathcal{L}\subseteq\mathfrak{d}_j\subseteq\mathcal{M}_\mathfrak{D}$, so $\eta^*_{I_j}=\eta^*_{I_j+M_N}$ for all $j\in\{1,\ldots,s_N\}$. By \eqref{etyprox}, $S_{M_N}\eta|_{F_N}=S_{M_N}\eta^*|_{F_N}$, so $d(S_{M_N}\eta,S_{M_N}\eta^*)\leq\frac{1}{2^N}$. Hence $(\eta,\eta^*)$ is proximal.
\end{proof}
To prove Theorem B, we will also need the following lemma.
\begin{Lemma}\label{Toeplitz_SK}
If $\eta$ is an $\OK$-Toeplitz array, then $\eta\in Y$.
\end{Lemma}
\begin{proof}[Proof of Theorem \ref{equivalent}]
Obviously, \eqref{top}$\Rightarrow$\eqref{min}.

\eqref{min}$\Rightarrow$\eqref{top}. It follows from~\eqref{min} and from Theorem A that $X_{\eta^*}=X_\eta$. Suppose that $\eta^*<\eta$, i.e. there is $n\in\OK$ such that $0=\eta^*_n<\eta_n=1$. Let $\mathfrak{p}$ be such that $\eta^*|_{n+\mathfrak{p}}=0$. Notice that both, $\Per(\eta,\mathfrak{p},0)$ and $\Per(\eta^*,\mathfrak{p},0)$ are invariant under translations by the elements of $\mathfrak{p}$, so it is natural to look at them as subsets of the finite quotient group $\OK/\mathfrak{p}$. It follows from~\eqref{downar_mul} that
\begin{equation}\label{jo1}
|\Per(\eta^*,\mathfrak{p},0)\bmod \mathfrak{p}|=|\Per(\eta,\mathfrak{p},0)\bmod \mathfrak{p}|
\end{equation}
as $\eta\in X_{\eta^*}$. On the other hand, since $\eta^*\leq \eta$, it follows immediately that
\[
\Per(\eta^*,\mathfrak{p},0)\supseteq \Per(\eta,\mathfrak{p},0).
\]
It remains to notice that $n\in\Per(\eta^*,\mathfrak{p},0)\setminus \Per(\eta,\mathfrak{p},0)$, which contradicts~\eqref{jo1}. This yields $\eta=\eta^*$, i.e.~\eqref{top} holds.

\eqref{top} $\Rightarrow$ \eqref{Y} Suppose $\eta$ is an $\OK$-Toeplitz array. For any $\mathfrak{b}\in\mathfrak{B}$ and~$0\leq s_\mathfrak{b}\leq N(\mathfrak{b})$ let \[Y^\mathfrak{b}_{\geq s_{\mathfrak{b}}}=\{x\in \{0,1\}^{\OK};\ |\text{supp }x\bmod \mathfrak{b}|\leq N(\mathfrak{b})-s_\mathfrak{b}\}\] and \[Y^\mathfrak{b}_{s_\mathfrak{b}}=\{x\in \{0,1\}^{\OK};\ |\text{supp }x\bmod \mathfrak{b}|=N(\mathfrak{b})-s_\mathfrak{b}\}.\] Then $Y^\mathfrak{b}_{\geq s_{\mathfrak{b}}}$ is closed and ~$(S_g)_{g\in\OK}$-invariant. By the minimality of $(X_\eta,(S_g)_{g\in\OK})$, it follows that for any $0\leq s_\mathfrak{b}\leq N(\mathfrak{b})$, $\mathfrak{b}\in\mathfrak{B}$, we have
\[
X_\eta \cap Y^\mathfrak{b}_{\geq s_{\mathfrak{b}}}=X_\eta\text{ or }X_\eta \cap Y^\mathfrak{b}_{\geq s_{\mathfrak{b}}}=\emptyset.
\]
By Lemma~\ref{Toeplitz_SK}, we have $\eta\in Y$, so for $s_\mathfrak{b}\geq 2$ we have \[X_\eta \cap Y^\mathfrak{b}_{\geq s_\mathfrak{b}}=\emptyset.\] Because for any $\mathfrak{b}\in\mathfrak{B}$ \[X_\eta=X_\eta \cap \left(\bigcup_{1\leq s_\mathfrak{b}\leq N(\mathfrak{b})}Y^\mathfrak{b}_{s_\mathfrak{b}}\right),\] it follows immediately that $X_\eta=X_\eta\cap Y$.

\eqref{Y} $\Rightarrow$ \eqref{top} Suppose $X_\eta\subseteq Y$. By Theorem~\ref{minimal2}, we have $X_{\eta^*}\subseteq X_\eta	$, so, in particular, $\eta^*\in X_\eta$ is an element of $Y$. Since $\eta^*$ is an $\OK$-Toeplitz array, it follows by  Lemma~\ref{Toeplitz_SK} that
\[
\eta^*\in Y^*,
\]
where
\[
Y^*=\{x\in\{0,1\}^{\OK};\ |\supp \ x\bmod \mathfrak{b}^*|=N(\mathfrak{b}^*)-1 \text{ for any }\mathfrak{b}^*\in\mathfrak{B}^*\}.
\]
Suppose there exists $\mathfrak{b}\in\mathfrak{B}\setminus\mathfrak{B}^*$. Let $\mathfrak{b}^*\in\mathfrak{B}$ be such that $\mathfrak{b}\subseteq\mathfrak{b}^*$. Then \[|\supp \ \eta^* \bmod \mathfrak{b}^*|=N(\mathfrak{b}^*)-1
\]
and we conclude that
\[
\left|\supp \ \eta^* \bmod \mathfrak{b}\right|=\frac{N(\mathfrak{b})}{N(\mathfrak{b}^*)}\left|\supp \ \eta^* \bmod \mathfrak{b}^*\right|=\frac{N(\mathfrak{b})}{N(\mathfrak{b}^*)}\left(N(\mathfrak{b}^*)-1\right)<N(\mathfrak{b})-1.
\]
Thus, $\eta^*\not\in Y$, which we know that is not true. It follows that $\mathfrak{B}=\mathfrak{B}^*$ and therefore~$\eta=\eta^*$ is an $\OK$-Toeplitz array.
\end{proof}

\begin{proof}[Proof of Lemma~\ref{Toeplitz_SK}]
Assume that $\eta$ is an $\OK$-Toeplitz array. Suppose $\eta\not\in Y$. Then for some $\mathfrak{b}\in\mathfrak{B}$ and some $r\not\in\mathfrak{b}$
\[r+\mathfrak{b}\subseteq\cm_\mathfrak{B}.\]
By Theorem \ref{Thm_SK}, there exists $\mathfrak{b}'\in\mathfrak{B}$ such that
\[(r)+\mathfrak{b}\subseteq\mathfrak{b}'.\]
Then by the primitivity of $\mathfrak{B}$, $\mathfrak{b}=(r)+\mathfrak{b}=\mathfrak{b}'$, which contradicts $r\not\in\mathfrak{b}$.
\end{proof}

\section{Counterexample to Theorem~\ref{equivalent} for lattices}\label{se5}
We mentioned in the introduction that as the counterpart of the set of multiples $\cM_{\sB}$ one can choose countable unions of~sublattices of~$\Z^m$. In this case the analogous result to Theorem \ref{equivalent} does not hold. Another example in \cite{MR4251829} shows that the arithmetic characterization of proxmiality valid in the setting of ideals in integer rings also fails to hold in the lattice setting. These examples can be treated as a reason to work with integer rings in number fields and ideals rather than with lattices.

Let us state now two necessary results.
\begin{Th}[{\cite[Proposition 3.17]{MR4251829}}]\label{prime}
Let $m\geq2$ and $\{\Lambda_i\}_{i\geq1}$ be an inifite and pairwise coprime collection of lattices in $\Z^m$.
Then $\left\{[\Z^m:\Lambda_i]\right\}_{i\geq1}$ contains an infinite pairwise coprime subset.
\end{Th}
\begin{Lemma}\label{prostokatna}
Let $\Lambda=(a,b)\Z+(0,d)\Z$ be a lattice in~$\Z^2$. Then the maximal (with respect to the inclusion relation) lattice contained in~$\Lambda$ equals \[\frac{ad}{\gcd(b,d)}\Z\times d\Z.\]
\end{Lemma}
The proof of the above lemma is straightforward, so we skip it.
\begin{Example}\label{ex_lattices}
Let $a_i,c_i\in\N\setminus\{1\}$, $i\geq1$ be such that $\{a_ic_i\}_{i\geq1}$ is an infinite and pairwise coprime set of odd numbers, $\{a_i\}_{i\geq1}$ and~$\{c_i\}_{i\geq1}$ are infinite.
Let \[\Lambda_i=(a_i,2^i)\Z+(0,2^ic_i)\Z, i\geq1\] and \[\mathfrak{B}=\{\Lambda_i\}_{i\geq1}.\] We claim that:
\begin{enumerate}[1)]
    \item $\eta=\raz_{\Z^2\setminus\bigcup_{i\geq1}\Lambda_i}$ is not a $\Z^2$-Toeplitz array,
    \item $\mathfrak{B}$ does not contain a scaled copy of infinite pairwise coprime collection of lattices $\{\Lambda_i'\}_{i\geq1}$,
    i.e.\ $\{(a,b)\Lambda_i'\}_{i\geq1}\nsubseteq\mathfrak{B}$ for any $(a,b)\in\Z^2$, such that $a,b\neq0$, where we consider the coordinatewise multiplication in $\Z^2$.
\end{enumerate}

Indeed, by Lemma~\ref{prostokatna} (for $\Lambda=\Lambda_i$, $i\geq1$), for $i\geq1$ \[a_ic_i\Z\times 2^ic_i\Z\subseteq\Lambda_i.\] Since $\{a_ic_i\}_{i\geq1}$ is infinite and pairwise coprime, there are arbitrarly long blocks of consecutive zeros on $\eta_1:=\left(\eta_{(n,0)}\right)_{n\in\Z}$. Since $a_i\geq2$, $i\geq1$, $\eta_{(1,0)}=1$. Suppose that there is $\Lambda=(c,d)\Z+(0,f)\Z$ such that \[(1,0)+\Lambda\subseteq\cf_\mathfrak{B}.\] By Lemma~\ref{prostokatna},
\[
\frac{cf}{\gcd(d,f)}\Z\times d\Z\subseteq \Lambda.
\]
Hence, for any $k\in\Z$,
\[
\eta_1\left(1+\frac{cf}{\gcd(d,f)}k\right)=1,
\]
but this contradicts that there are arbitrarly long blocks of consecutive zeros on $\eta_1$. Hence $\eta$ is not a $\Z^2$-Toeplitz array.

Suppose that for some infinite pairwise coprime collection of lattices $\{\Lambda_i'\}_{i\geq1}$ and some $(a,b)\in\Z^2$ we have
\[\{(a,b)\Lambda_i'\}_{i\geq1}\subseteq\mathfrak{B}.\] Since $(a,b)\Lambda_i'\in\mathfrak{B}$, there exists $j_i\geq1$ such that \begin{equation}\label{L_ji}
(a,b)\Lambda_i'=\Lambda_{j_i}.
\end{equation} Let $\Lambda_i'=(d_i',e_i')\Z+(0,f_i')\Z$, $i\geq1$. Since $\{\Lambda_i'\}_{i\geq1}$  is a collection of pairwise coprime lattices, by Theorem~\ref{prime} we can assume $\{d_i'f_i'\}_{i\geq1}$ is infinite and pairwise coprime. Let $r_i=d_i'f_i'$, $i\geq1$. Since $r_i=[\Z^2:\Lambda_i']$, we have $r_i\Z^2\subseteq\Lambda_i'$ for any $i\geq1$, so \[ar_i\Z\times br_i\Z\subseteq\Lambda_{j_i}.\] By Lemma~\ref{prostokatna}, the maximal lattice contained in~$\Lambda_{j_i}$ equals \[a_{j_i}c_{j_i}\Z\times 2^{j_i}c_{j_i}\Z.\] Hence
\[
ar_i\Z\times br_i\Z\subseteq a_{j_i}c_{j_i}\Z\times 2^{j_i}c_{j_i}\Z,\ i\geq1.
\]
Thus, \[2^{j_i}c_{j_i}\mid br_i, i\geq1.\] Since $r_i$, $i\geq1$, are pairwise coprime, we can assume that all of them are odd. Hence \[2^{j_i}\mid b,\] but $\{\Lambda_i'\}_{i\geq1}$ is infinite, so $2^{j_i}\to\infty$ as $i\to\infty$. It follows that $b=0$, but then $(a,b)\Lambda_i'$ is of infinite index in~$\Z^2$, which contradicts \eqref{L_ji}, as elements of $\mathfrak{B}$ have finite indices.
\end{Example}
\bibliographystyle{acm}
\bibliography{basic}
        \bigskip
        \footnotesize

		\noindent
		Aurelia Dymek\\
 \textsc{Faculty of Mathematics and Computer Science, Nicolaus Copernicus University,\\ Chopina 12/18, 87-100 Toru\'{n}, Poland}\par\nopagebreak
        \noindent
        \href{mailto:aurbart@mat.umk.pl}
        {\texttt{aurbart@mat.umk.pl}}

\medskip

		\noindent
		Stanis\l aw Kasjan\\
        \textsc{Faculty of Mathematics and Computer Science, Nicolaus Copernicus University,\\ Chopina 12/18, 87-100 Toru\'{n}, Poland}\par\nopagebreak
        \noindent
        \href{mailto:skasjan@mat.umk.pl}
        {\texttt{skasjan@mat.umk.pl}}

\medskip        
        
        \noindent       
        Joanna Ku\l aga-Przymus\\
        \textsc{Faculty of Mathematics and Computer Science, Nicolaus Copernicus University,\\ Chopina 12/18, 87-100 Toru\'{n}, Poland}\par\nopagebreak
        \noindent
        \href{mailto:joasiak@mat.umk.pl}
        {\texttt{joasiak@mat.umk.pl}}

\end{document}